\documentclass[12pt]{amsart} 
\usepackage{amsfonts,graphics,amsmath,amsthm,amsfonts,amscd,amssymb,amsmath,latexsym,multicol,euscript}
\usepackage{enumerate}
\usepackage{epsfig,url}
\usepackage{flafter}
\usepackage[active]{srcltx}
\usepackage{hyperref}
\usepackage[small,nohug,UglyObsolete]{diagrams}
\diagramstyle[labelstyle=\scriptstyle]
\diagramstyle[noPostScript]

\makeatletter

\def\jobis#1{FF\fi
  \def\preedicate{#1}%
  \edef\preedicate{\expandafter\strip@prefix\meaning\preedicate}%
  \edef\job{\jobname}%
  \ifx\job\preedicate
}

\makeatother

\if\jobis{proposal}%
 \def\try{subsection}%
\else
  \def\try{section}%
\fi

\theoremstyle{plain}
\newtheorem{theorem}{Theorem}[\try]

\newtheorem{lemma}[theorem]{Lemma}

\newtheorem{example}[theorem]{Example}
\newtheorem{proposition}[theorem]{Proposition}
\newtheorem{definition-lemma}[theorem]{Definition-Lemma}
\newtheorem{definition-proposition}[theorem]{Definition-Proposition}
\newtheorem{question}[theorem]{Question}
\newtheorem{definition}[theorem]{Definition}
\newtheorem{remark}[theorem]{Remark}

\newtheorem{conjecture}[theorem]{Conjecture}

\def\scr#1{\mathbf{\EuScript{#1}}}

\def\lfomitlist#1.#2.#3.#4.{{#1}_0,{#1}_1 #2 \dots #2\hat{{#1}_{#4}} #2\dots #2 {#1}_{#3}}

\def\alist#1.#2.#3.{{#1}_1 #2 {#1}_2 #2\dots #2 {#1}_{#3}}
\def\zlist#1.#2.#3.{#1_0 #2 #1_1 #2\dots #2 #1_{#3}}
\def\ltomitlist#1.#2.#3.{{#1}_0,{#1}_1 #2 \dots #2\hat {{#1}_i} #2\dots #2 {#1}_{#3}}
\def\lomitlist#1.#2.#3.{{#1}_0 #2 {#1}_1 #2 \dots #2 \hat {{#1}_i} #2 \dots #2 {#1}_{#3}}
\def\lmap#1.#2.#3.{#1 \overset{#2}{\longrightarrow} #3}
\def\mes#1.#2.#3.{#1 \longrightarrow #2 \longrightarrow #3}
\def\ses#1.#2.#3.{0\longrightarrow #1 \longrightarrow #2 \longrightarrow #3 \longrightarrow 0}
\def\les#1.#2.#3.{0\longrightarrow #1 \longrightarrow #2 \longrightarrow #3}
\def\res#1.#2.#3.{#1 \longrightarrow #2 \longrightarrow #3\longrightarrow 0}
\def\Hi#1.#2.#3.{\text {Hilb}^{#1}_{#2}(#3)}
\def\ten#1.#2.#3.{#1\underset {#2}{\otimes} #3}

\def\mderiv#1.#2.#3.{\frac {d^{#3} #1}{d #2^{#3}}}
\def\mfderiv#1.#2.#3.{\frac {\partial^{#3} #1}{\partial #2}}
\def\ggr#1.#2.#3.{\mathbb{G}_{#1}(#2,#3)}

\def\llist#1.#2.{{#1}_1,{#1}_2,\dots,{#1}_{#2}}
\def\ulist#1.#2.{{#1}^1,{#1}^2,\dots,{#1}^{#2}}
\def\lomitlist#1.#2.{{#1}_1,{#1}_2,\dots,\hat {{#1}_i}, \dots, {#1}_{#2}}
\def\lomitlistz#1.#2.{{#1}_0,{#1}_1,\dots,\hat {{#1}_i}, \dots, {#1}_{#2}}
\def\loc#1.#2.{\Cal O_{#1,#2}}
\def\fderiv#1.#2.{\frac {\partial #1}{\partial #2}}
\def\deriv#1.#2.{\frac {d #1}{d #2}}
\def\map#1.#2.{#1 \longrightarrow #2}
\def\rmap#1.#2.{#1 \dasharrow #2}
\def\emb#1.#2.{#1 \hookrightarrow #2}
\def\non#1.#2.{\text {Spec }#1[\epsilon]/(\epsilon)^{#2}}
\def\Hi#1.#2.{\text {Hilb}^{#1}(#2)}
\def\sym#1.#2.{\operatorname {Sym}^{#1}(#2)}
\def\Hb#1.#2.{\text {Hilb}_{#1}(#2)}
\def\Hm#1.#2.{\Hom_{#1}(#2)}
\def\prd#1.#2.{{#1}_1\cdot {#1}_2\cdots {#1}_{#2}}
\def\Bl #1.#2.{\operatorname {Bl}_{#1}#2}
\def\pl #1.#2.{#1^{\otimes #2}}
\def\mgn#1.#2.{\overline {M}_{#1,#2}}
\def\ialist#1.#2.{{#1}_1 #2 {#1}_2 #2 {#1}_3 #2\dots}
\def\pair#1.#2.{\langle #1, #2\rangle}
\def\gproj#1.#2.{\mathbb{P}_{#1}(#2)}
\def\gpr #1.#2.{\mathbb{P}^{#1}_{#2}}
\def\gaf #1.#2.{\mathbb{A}^{#1}_{#2}}
\def\vandermonde#1.#2.{\left|
\begin{matrix}
1 & 1 & 1 & \dots & 1\\
{#1}_1 & {#1}_2 & {#1}_3 & \dots & {#1}_{#2}\\
{#1}_1^2 & {#1}_2^2 & {#1}_3^2 & \dots & {#1}_{#2}^2\\
\vdots & \vdots & \vdots & \ddots & \vdots\\
{#1}_1^{#2-1} & {#1}_2^{#2-1} & {#1}_2^{#2-1} & \dots & {#1}_{#2}^{#2-1}\\
\end{matrix}
\right|
}
\def\vandermondet#1.#2.{\left|
\begin{matrix}
1 & {#1}_1   & {#1}_1^2 & \dots & {#1}_1^{#2-1}\\
1 & {#1}_2   & {#1}_2^2 & \dots & {#1}_2^{#2-1}\\
1 & {#1}_3   & {#1}_3^2 & \dots & {#1}_3^{#2-1}\\
\vdots & \vdots & \vdots & \ddots & \vdots\\
1 & {#1}_{#2}& {#1}_{#2}^2 & \dots & {#1}_{#2}^{#2-1}\\
\end{matrix}
\right|
}
\def\gr#1.#2.{\mathbb{G}(#1,#2)}

\def\ideal#1.{I_{#1}}
\def\ring#1.{\mathcal {O}_{#1}}
\def\fring#1.{\hat{\mathcal {O}}_{#1}}
\def\proj#1.{\mathbb {P}(#1)}
\def\pr #1.{\mathbb {P}^{#1}}
\def\dpr #1.{\hat{\mathbb {P}}^{#1}}
\def\af #1.{\mathbb{A}^{#1}}
\def\Hz #1.{\mathbb{F}_{#1}}
\def\Hbz #1.{\overline{\mathbb {F}}_{#1}}
\def\fb#1.{\underset {#1} {\times}}
\def\rest#1.{\underset {\ \ring #1.} \to \otimes}
\def\au#1.{\operatorname {Aut}\,(#1)}
\def\deg#1.{\operatorname {deg } (#1)}
\def\pic#1.{\operatorname {Pic}\,(#1)}
\def\pico#1.{\operatorname{Pic}^0(#1)}
\def\picg#1.{\operatorname {Pic}^G(#1)}
\def\ner#1.{NS (#1)}
\def\rdown#1.{\llcorner#1\lrcorner}
\def\rfdown#1.{\lfloor{#1}\rfloor}
\def\rup#1.{\ulcorner{#1}\urcorner}
\def\rfup#1.{\lceil{#1}\rceil}
\def\bp#1.{#1^{{}\leq 1}}
\def\rcup#1.{\lceil{#1}\rceil}
\def\cone#1.{\operatorname {NE}(#1)}
\def\mone#1.{\operatorname {NM}(#1)}
\def\ccone#1.{\overline{\operatorname {NE}}(#1)}
\def\cmone#1.{\overline{\operatorname {NM}}(#1)}
\def\none#1.{\operatorname {NF}(#1)}
\def\cnone#1.{\overline{\operatorname {NF}}(#1)}
\def\coef#1.{\frac{(#1-1)}{#1}}
\def\vit#1.{D_{\langle #1 \rangle}}
\def\mm#1.{\overline {M}_{0,#1}}
\def\Hone#1.{H^1(#1,{\ring #1.})}
\def\ac#1.{\overline {\mathbb F}_{#1}}
\def\dir#1{\overset\rightarrow\to{#1}}

\def\adj#1.{\frac {#1-1}{#1}}
\def\spn#1.{\overline{#1}}
\def\pek#1.#2.{\Cal P^{#1}(#2)}
\def\plk#1.#2.{\Cal P^{\leq #1}(#2)}
\def\ev#1.{\operatorname{ev_{#1}}}
\def\ilist#1.{{#1}_1,{#1}_2,\ldots}
\def\bminv#1.{(\nu_1,s_1;\nu_2,s_2;\dots ;\nu_{#1},s_{#1};\nu_{r+1})}
\def\zinv#1.{(\nu_1,s_1;\nu_2,s_2;\dots ;\nu_{#1},s_{#1};0)}
\def\iinv#1.{(\nu_1,s_1;\nu_2,s_2;\dots ;\nu_{#1},s_{#1};\infty)}
\def\scr#1.{\mathbf{\EuScript{#1}}}
\def\mg#1.{\overline {M}_{#1}}
\def\inter#1.{\underset #1{\cdot}}
\def\cate#1.{\text{(\underline{#1})}}
\def\dir#1.{\overrightarrow{#1}}

\def\Hom{\operatorname{Hom}}

\def\dim{\operatorname{dim}}

\def\deg{\operatorname{deg}}

\def\Aut{\operatorname{Aut}}

\def\mult{\operatorname{mult}}

\def\rest{\operatorname{res}}

\def\vol{\operatorname{vol}}

\def\C`har{\operatorname{char}}

\def\dir{\operatorname{div}}

\def\C{\mathbb C}

\def\e{\Cal E}

\def\e1{E_1}
\def\e2{E_2}

\def\ds{\displaystyle}

\def\mapdown#1{\big\downarrow\rlap{$\vcenter
{\hbox{$\scriptstyle#1$}}$}}

\def\mapse#1{
{\vcenter{\hbox{$\mathop{\smash{\raise1pt\hbox{$\diagdown$}\!\lower7pt
\hbox{$\searrow$}}\vphantom{p}}\limits_{#1}\vphantom{\mapdown{}}$}}}}

\def\VR#1.{height#1pt&\omit&&\omit&&\omit&&\omit&&\omit&\cr}

\def\VRT#1.{height#1pt&\omit&&\omit&\cr}

\begin{document}
\title[Automorphisms of varieties]{On the birational automorphisms of varieties of general type} 
\author{Christopher D. Hacon}
\date{\today}
\address{Department of Mathematics \\
University of Utah\\
155 South 1400 East\\
JWB 233\\
Salt Lake City, UT 84112, USA}
\email{hacon@math.utah.edu}
\author{James M\textsuperscript{c}Kernan}
\address{Department of Mathematics\\
MIT\\
77 Massachusetts Avenue\\
Cambridge, MA 02139, USA}
\email{mckernan@math.mit.edu}
\author{Chenyang Xu}
\address{Beijing International Center of Mathematics Research\\ 5 Yiheyuan Road, Haidian District\\ 
Beijing 100871, China}
\email{cyxu@math.pku.edu.cn}
\address{Department of Mathematics \\University of Utah\\ 
155 South 1400 East \\ 
JWB 233\\
Salt Lake City, UT 84112, USA}
\email{cyxu@math.utah.edu}

\dedicatory{In memory of Eckart Viehweg}

\thanks{The first author was partially supported by NSF research grant no: 0757897, the
  second author was partially supported by NSF research grant no: 0701101, and the third
  author was partially supported by NSF research grant no: 0969495 and by a special
  research fund in China.  We would like to thank Xinyi Yuan for sparking our initial
  interest in this problem, J\'anos Koll\'ar for many useful conversations about this
  paper, and Igor Dolgachev and Mihai P{\u a}un for some helpful suggestions.  We would
  also like to thank the referee for some helpful comments}

\begin{abstract} We show that the number of birational automorphisms of a variety of
general type $X$ is bounded by $c \cdot \vol(X,K_X)$, where $c$ is a constant which only
depends on the dimension of $X$.
\end{abstract}

\maketitle

\tableofcontents

\section{Introduction}

Throughout this paper, unless otherwise mentioned, the ground field $k$ will be an
algebraically closed field of characteristic zero.  
\begin{theorem}\label{t_boundauto} If $n$ is a positive integer, then there is a constant
$c$ such that the birational automorphism group of any projective variety $X$ of general type of
dimension $n$ has at most $c\cdot\vol(X,K_X)$ elements.
\end{theorem}

For curves, this is a weak form of the classical Hurwitz Theorem which says that if $C$ is
a curve of genus $g\geq 2$ with automorphism group $G$, then $|G|\leq 84(g-1)$.  Note that
$\vol(C,K_C)=2g-2$ and so this bound may be rephrased as $|G|\leq 42\cdot \vol(C,K_C)$.

This problem has been extensively studied in higher dimensions, see for example,
\cite{Alexeev94}, \cite{Andreotti50}, \cite{Cai00}, \cite{Corti91}, \cite{HS91},
\cite{Xiao94}, and \cite{Xiao95}, for surfaces, \cite{CS95}, \cite{Szabo96},
\cite{Tsuji00}, \cite{Xiao96}, and \cite{Zhang07}, in higher dimensions, and
\cite{Ballico93}, for surfaces in characteristic $p$.

Xiao, \cite{Xiao95}, proved that if $S$ is a smooth projective surface of general type,
with automorphism group $G$, then $|G|\leq (42)^2\vol(S,K_S)$ (if $S$ is minimal, then
$\vol(S,K_S)=K^2_S$; for the general definition of the volume, see
\cite[2.2.31]{Lazarsfeld04b} or \eqref{d_volume}).  Xiao shows that we have equality if
and only if $S$ is a quotient of $C\times C$, where $C$ is a curve whose automorphism
group has cardinality $42(2g-2)$, by the action of a very special subgroup of the
automorphism group of $C\times C$.
\begin{question} Find an explicit bound for the constant $c$ appearing in
\eqref{t_boundauto}.  
\end{question}

If $C$ is a curve with automorphism group of maximal size, that is,
$|\Aut(C)|=84(g-1)$ and 
$$
X=C\times C\times\cdots\times C,
$$
then $\Aut(X)=n!(42)^n(2g-2)^n$ and $\vol(X,K_X)=n!(2g-2)^n$, so that $c\geq 42^n$.  If we consider
the example of the Fermat hypersurface
$$
X=(\zlist X^m.+.n+1.=0)\subset \pr n+1.,
$$
then $\Aut(X)\geq (n+2)!m^{n+1}$ and $\vol(X,K_X)=m(m-n-2)^n$.  If we take $m=n+3$ then the ratio 
$$
\frac{\Aut(X)}{\vol(X,K_X)}\ge(n+2)!(n+3)^{n},
$$
exceeds $42^n$ for $n$ sufficiently large (indeed, $n\geq 5$ suffices), so that $c$ is
eventually greater than $42^n$.  In fact, $c$ grows faster than $n^n$, so that $c$ grows
faster than any exponential function.

It is all too easy to give examples which show that \eqref{t_boundauto} fails spectacularly
in characteristic $p$.  Consider the finite field $\mathbb{F}_{q^2}$ with $q^2$ elements,
where $q=p^k$ is a power of a prime $p$.  Note that the function
$$
\map \mathbb{F}_{q^2}.\mathbb{F}_{q^2}. \qquad \text{given by} \qquad \map b.\bar b=b^q.,
$$
is an involution which plays the role of complex conjugation in characteristic $p$.
Suppose that $V=\mathbb{F}_{q^2}^m$ is the standard vector space of dimension $m$ over the
field $\mathbb{F}_{q^2}$.  Then there is a sesquilinear pairing 
$$
\map V\times V.\mathbb{F}_{q^2}. \qquad \text{given by} \qquad \map (a,b).\sum_i a_i\bar b_i..
$$
Let $U_m(q)$ denote the group of $m\times m$ unitary matrices over the field
$\mathbb{F}_{q^2}$, so that $U_m(q)$ is the group of linear maps of $V$ preserving the
pairing.  Recall that $U_m(q)$ is a finite simple group of Lie type, see for example
\cite{GLS94}, whose notation we follow.  Note that the Fermat hypersurface
$$
X=(\zlist X^{q+1}.+.n+1.=0)\subset \pr n+1.,
$$
is the projectivisation of the null cone of the pairing, so that $\Aut(X)\supset
U_{n+2}(q)$.  We have 
$$
|U_{n+2}(q)|=\frac 1{(n+2,q+1)}q^{\binom{n+2}2}\prod_{i=2}^{n+2} (q^i-(-1)^i),
$$
see, for example, the table on page 8 of \cite{GLS94}.  Note that both the order of the
automorphism group and the volume of the Fermat hypersurface are polynomials $f$ and $g$
in $q$.  $f$ has degree 
$$
\binom{n+2}2+\binom{n+3}2-1,
$$
and $g$ has degree 
$$
(n+1).
$$
If $n=1$, the genus is a quadratic polynomial in $q$ and the order of the automorphism
group is bounded by a polynomial of degree $4$ in $g$.  
\begin{question}\label{q_all} Fix a positive integer $n$.  Can we find positive integers
$c$ and $d$ such that if $X$ is any $n$-dimensional smooth projective variety of general type over an
algebraically closed field of arbitrary characteristic, then
$$
|\Aut(X)|\leq c\cdot \vol(X,K_X)^d?
$$
\end{question}

It is known that if $n=1$ then we may take $c=216$ and $d=4$ (cf. \cite{Stichtenoth73}).  

We now explain how to derive \eqref{t_boundauto} from a result about the quotient.  If $Y$
is a variety of general type, then the automorphism group $G=\Aut(Y)$ is known to be
finite, see \cite{Matsumura63}.  If $f\colon\map Y.X=Y/G.$ is the quotient map, then there
is a $\mathbb{Q}$-divisor $\Delta$ on $X$ such that $K_Y=f^*(K_X+\Delta)$.  We call any
such log pair $(X,\Delta)$ a \textit{global quotient}, cf. \eqref{d_global-quotient}.  As
$$
\vol(Y,K_Y)=|G|\cdot \vol(X,K_X+\Delta),
$$
the main issue is to bound $\vol(X,K_X+\Delta)$ from below:
\begin{theorem}\label{t_volume} Fix a positive integer $n$.  Let $\mathfrak{D}$ be the set 
of log pairs $(X,\Delta)$, which are global quotients, where $X$ is projective of
dimension $n$.

\begin{enumerate}
\item the set 
$$
\{\, \vol(X,K_X+\Delta) \,|\, (X,\Delta)\in \mathfrak{D} \,\},
$$
satisfies the DCC.  
\end{enumerate}
Further, there are two constants $\delta>0$ and $M$ such that  if $(X,\Delta)\in \mathfrak{D}$ and $K_X+\Delta$ is big, then 
\begin{enumerate}
\setcounter{enumi}{1}
\item $\vol(X,K_X+\Delta)\geq \delta$, and 
\item $\phi_{M(K_X+\Delta)}$ is birational.
\end{enumerate}
\end{theorem}

DCC is an abbreviation for the descending chain condition.  Note that, by convention,
$\phi_{M(K_X+\Delta)}=\phi_{\rfdown M(K_X+\Delta).}$ (cf. \eqref{ss-NC}).  Note also that the set of volumes
of smooth projective varieties of fixed dimension is a discrete set (cf. \cite{Tsuji07},
\cite{HM05b} and \cite{Takayama06}).  The situation for log pairs is considerably more
subtle.

\begin{remark} \cite{Tsuji07}, \cite{HM05b} and \cite{Takayama06} show that if we fix a
positive integer $n$, then the set 
$$
\{\, \vol(X,K_X) \,|\, \text{$X$ is a smooth projective variety of dimension $n$} \,\},
$$
is discrete.  However, the corresponding statement fails for kawamata log terminal
surfaces, whence also for surface with reduced boundary with simple normal crossings.  See
\cite{Kollar08} for an example.
\end{remark}

However we do have:
\begin{conjecture}[Koll\'ar, cf. \cite{Kollar92b}, \cite{Alexeev94}]\label{c_general} Fix $n\in \mathbb{N}$ and a set 
$I\subset [0,1]$ which satisfies the DCC.

If $\mathfrak{D}$ is the set of simple normal crossings pairs $(X,\Delta)$, where $X$ is
projective of dimension $n$, and the coefficients of $\Delta$ belong to $I$, then the set
$$
\{\, \vol(X,K_X+\Delta) \,|\, (X,\Delta)\in \mathfrak{D} \,\},
$$
satisfies the DCC.
\end{conjecture}

Alexeev, cf. \cite{Alexeev94} and \cite{AM04}, proved \eqref{c_general} for surfaces.
Note that if $(X,\Delta)$ is a global quotient, then the coefficients of $\Delta$ belong
to the set $I=\ds{\{\, \frac{r-1}r \,|\, r\in \mathbb{N}\,\}}$, so that \eqref{t_volume} is a 
special case of \eqref{c_general}.  

We hope to give an affirmative answer to \eqref{c_general}, using some of the techniques
developed in this paper.  Let
$$
I=\{\, \frac{r-1}r \,|\, r\in \mathbb{N} \,\}.
$$
Assuming an affirmative answer to \eqref{c_general} for this particular set, it is
interesting to wonder what is the smallest possible volume.  Let
$$
(X,\Delta)=(\pr n.,\frac 12H_0+\frac 23H_1+\frac 67H_2+\dots +\frac {r_{n+1}}{r_{n+1}+1}H_{n+1}),
$$
where $\zlist H.,.{n+1}.$ are $n+2$ general hyperplanes and $\ilist r.$ are defined recursively
by:
$$
r_0=1 \qquad \text{and} \qquad r_{n+1}=r_n(r_n+1).
$$
Note that $(X,\Delta)\in \mathfrak{D}$.  It is easy to see that the volume of $K_X+\Delta$
is
$$
\frac 1{r_{n+2}^n}.
$$
\begin{question}\label{c_general-optimal} Find an explicit bound for
$$
\delta=\min \{\, \vol(X,K_X+\Delta) \,|\, (X,\Delta)\in \mathfrak{D} \,\}.
$$
\end{question}

The most optimistic answer to \eqref{c_general-optimal} would be
$$
\delta=\frac 1{r_{n+2}^n}.  
$$

Note that $c\geq \frac 1{\delta}$.  When $n=1$, we have
$$
\delta=\frac 1{r_3}=\frac 1{42},
$$
and the reciprocal is precisely the constant $c=42$.  On the other hand, one can check 
that $r_n$ grows roughly like
$$
a^{2^n},
$$
for some constant $a>1$, so that in general there is a huge difference between $r_n$ and
$c_n$.  In fact it is easy to check that 
$$
r_{n+1}=\prod_{i=0}^n (r_i+1),
$$
see \S 8 of \cite{Kollar95} for more details.

\begin{theorem}[Deformation invariance of log plurigenera]\label{t_dilp} Let 
$\pi\colon\map X.T.$ be a projective morphism of smooth varieties.  Suppose that
$(X,\Delta)$ is log canonical and has simple normal crossings over $T$.
\begin{enumerate}
\item If $K_X+\Delta$ is kawamata log terminal, and either $K_X+\Delta$ or $\Delta$ is big
over $T$, and $m$ is any positive integer such that $m\Delta$ is integral, then
$h^0(X_t,\ring X_t.(m(K_{X_t}+\Delta_t)))$ is independent of $t\in T$.
\item $\kappa_{\sigma}(X_t,K_{X_t}+\Delta_t)$ is independent of $t\in T$.  
\item $\vol(X_t,K_{X_t}+\Delta_t)$ is independent of $t\in T$.  
\end{enumerate} 
\end{theorem}

For the definitions of $\kappa_{\sigma}$ and simple normal crossings over $T$, see
\eqref{ss-NC}.  We will prove a similar but stronger statement \eqref{t_inv} which implies
\eqref{t_dilp}.  Obviously, \eqref{t_dilp} is a generalisation of Siu's theorem on
invariance of plurigenera, cf.  \cite{Siu98}.  We recently learnt that (1) of
\eqref{t_dilp} holds even without the assumption that $K_X+\Delta$ is big, see Theorem 0.2
of \cite{BP10}.  We use \eqref{t_dilp} to prove:
\begin{theorem}\label{t_dcc} Fix a set $I\subset [0,1]$ which satisfies the DCC.  
Let $\mathfrak{D}$ be a set of simple normal crossings pairs $(X,\Delta)$, which is log
birationally bounded (cf. \eqref{d_birationally-bounded}), such that if $(X,\Delta)\in
\mathfrak{D}$, then the coefficients of $\Delta$ belong to $I$.

Then the set
$$
\{\, \vol(X,K_X+\Delta) \,|\, (X,\Delta)\in \mathfrak{D} \,\},
$$
satisfies the DCC.  
\end{theorem}

\subsection{Sketch of the proof of \eqref{t_volume}}

The proof of \eqref{t_volume} is by induction on the dimension $n$ and the proof is
divided into two steps.  The first step uses some ideas of Tsuji which are used to prove
that some fixed multiple of $K_X$ defines a birational map, for a variety $X$ of general
type, see \cite{Tsuji07}, \cite{HM05b} and \cite{Takayama06}.  In this step we establish
that modified versions of (2) and (3) of \eqref{t_volume} are equivalent, given that
\eqref{t_volume}$_{n-1}$ holds (that is, \eqref{t_volume} holds in dimension $n-1$).
Namely, consider 
\begin{enumerate} 
\setcounter{enumi}{1}
\item $\vol(X,r(K_X+\Delta))>\delta$, and
\item $\phi_{Mr(K_X+\Delta)}$ is birational.
\end{enumerate} 
We show that if $(X,\Delta)$ is a global quotient of dimension $n$, then there are
constants $\delta>0$ and $M$ such that for every positive integer $r$, (2) implies (3),
see \eqref{t_tsuji}.

It is clear that if some fixed multiple of $r(K_X+\Delta)$ defines a birational map, then
the volume of $r(K_X+\Delta)$ is bounded from below, \eqref{l_standard}, so that there are
constants $\delta$ and $M$ such that (3) implies (2).  To go the other way, we need to
construct a divisor $0\leq D \sim_{\mathbb{Q}} mr(K_X+\Delta)$, where $m$ is fixed, which
has an isolated non kawamata log terminal centre at a very general point and is not
kawamata log terminal at another very general point, \eqref{l_potential}.  As we know that
log canonical models exist by \cite{BCHM06}, we may assume that $K_X+\Delta$ is ample, so
that lifting divisors from any subvariety is simply a matter of applying Serre vanishing.
In this case, it is well known, since the work of Anghern and Siu, \cite{AS95}, that to
construct $D$ we need to bound the volume of $K_X+\Delta$ from below on special
subvarieties $V$ of $X$ (specifically, any $V$ which is a non kawamata log terminal centre
of $(X,\Delta+\Delta_0)$, where $\Delta_0$ is proportional to $K_X+\Delta$), see
\eqref{t_potential} and \eqref{t_recursive}.

If $V=X$, there is nothing to prove, as we are assuming that
$\vol(X,r(K_X+\Delta))>\delta$.  Otherwise, the dimension of $V$ is less than the
dimension of $X$, and we may proceed by induction on the dimension; as $V$ passes through
a very general point of $X$, $V$ is birational to a global quotient.  In fact, we even
know that $\vol(V,(K_X+\Delta)|_V)$ is bounded from below.

So from now on we assume that (2) implies (3).  We may suppose that the constant $\delta$
appearing in (2) is at most one.  The next step is to prove that (3) holds when $r=1$
(that is, (3) of \eqref{t_volume} holds).  There are two cases.  If the volume of
$K_X+\Delta$ is at least one, then the volume of $K_X+\Delta$ is certainly bounded from
below, and there is nothing to prove.  Otherwise, we may find $r>0$ such that
$\delta<1\leq \vol(X,r(K_X+\Delta))<2^n$.  But then $\phi_{m(K_X+\Delta)}$ is birational,
where $m=Mr$ and, at the same time, the volume of $m(K_X+\Delta)$ is bounded from above.
In this case, the degree of the image of $\phi_{m(K_X+\Delta)}$ is bounded from above, and
so we know that the image belongs to a bounded family.  In fact, one can prove that both
the degree of the image and the degree of the image of $\Delta$ and the exceptional locus
have bounded degree, \eqref{t_boundingimage}, so we only need to concern ourselves with a
log birationally bounded family of log pairs, \eqref{d_birationally-bounded}.  This
finishes the first step.

To finish the argument, we need to argue that the volume is bounded from below if we have
a log birationally bounded family of log pairs.  This is the most delicate part of the
argument and is the second step.  We use some ideas which go back to Alexeev.  Firstly, it
is not much harder to prove that the volume of global quotients satisfies the DCC.  The
first part of the second step is to argue that we only need to worry about log pairs
$(X,\Delta)$ which are birational to a single pair $(Z,B)$, rather than a bounded family
of log pairs.  For this we prove a version of deformation invariance of log plurigenera
for log pairs, see \eqref{t_dilp}.  Deformation invariance fails in general
(cf. \cite[4.10]{FH11}); we need to assume that the family has simple normal crossings
over the base (which roughly means that every component of $\Delta$ is smooth over the
base).  To reduce to this case involves some straightforward manipulation of a family of
log pairs, see the proof of \eqref{t_dcc} in \S \ref{s_bound-volume}.  We prove,
cf. \eqref{t_inv}, a version of \eqref{t_dilp} which is better suited to induction; to
this end, we first show that if we have a family of log pairs over a curve then we can run
the MMP in a family, \eqref{p_mmp}.

So we are reduced to the most subtle part of the argument.  We are given a simple normal
crossings pair $(Z,B)$, a set $I$ which satisfies the DCC, and we want to argue that if
$(X,\Delta)$ is a simple normal crossings pair such that there is a birational morphism
$f\colon\map X.Z.$ with $f_*\Delta\leq B$, then the volume of $K_X+\Delta$ belongs to a
set which satisfies the DCC, see \eqref{p_limit}.  To fix ideas, let us suppose that $Z$
is a smooth surface (this is the case originally treated by Alexeev).  We are given a
sequence of simple normal crossings surfaces $(X_i,\Delta_i)$ and birational morphisms
$f_i\colon\map X_i.Z.$.

We have that $\Phi_i=f_{i*}\Delta_i\leq B$ and the coefficients of $\Phi_i$ belong to $I$.  Note
that we are free to pass to an arbitrary subsequence, so that we may assume that
$\Phi_i\leq \Phi_{i+1}$.  In particular, the volume of $K_Z+\Phi_i$ is not decreasing.
The problem is that if we write
$$
K_{X_i}+\Delta_i=f_i^*(K_Z+\Phi_i)+E_i,
$$
then $E_i$ might have negative coefficients, so that the volume of $K_{X_i}+\Delta_i$ is
less than the volume of $K_Z+\Phi_i$.  Suppose that we write $E_i=E_i^+-E_i^-$, where
$E_i^+\geq 0$ and $E_i^-\geq 0$ have no common components.  Note that $E_i^+$ has no
effect on the volume, as $E_i$ is supported on the exceptional locus.  What bothers us is
the possibility that the $E_i^-$ involve exceptional divisors that live on higher and
higher models.

Clearly, we should consider the limit $\Phi=\lim_i\Phi_i$.  However, this is not enough,
we need to take the limit of the divisors $\Delta_i$ on various models, and to work with
linear systems on these models.  It was for just this purpose that the language of
$\mathbf{b}$-divisors was introduced by Shokurov.  Recall that a $\mathbf{b}$-divisor
$\mathbf{D}$ is just the choice of a divisor $\mathbf{D}_X$ on every model $X$, which is
compatible under pushforward.

For us there are three relevant $\mathbf{b}$-divisors.  Since we want to work on higher
models without changing the volume, or the fact that the coefficients lie in the set $I$,
we introduce, \eqref{d_one-b-divisor}, the $\mathbf{b}$-divisor $\mathbf{M}_{\Delta_i}$
associated to a log pair $(X_i,\Delta_i)$: given a model $\map Y.X_i.$ we just throw in
any exceptional divisor with coefficient one.  We next take the limit $\mathbf{B}$ of the
sequence of $\mathbf{b}$-divisors $\mathbf{M}_{\Delta_i}$; on $Z$ we just recover the
divisor $\Phi$.  Finally, we define \eqref{d_associated-b-divisor}, a $\mathbf{b}$-divisor
$\mathbf{L}_{\Delta_i}$ that assigns to a model $\pi\colon\map Y.X_i.$ the positive part
of the log pullback.  Actually, this is the most complicated of the three
$\mathbf{b}$-divisors, and it is the subtle behaviour of the $\mathbf{b}$-divisors
$\mathbf{L}_{\Delta_i}$ which complicates the proof.

If we set $\Delta_i'=\Delta_i\wedge\mathbf{L}_{\Phi_i,X_i}$, then, 
$$
\vol(X_i,K_{X_i}+\Delta'_i)=\vol(X_i,K_{X_i}+\Delta_i),
$$
see (2) of \eqref{l_simple}.  If we knew that the coefficients of $\Delta'_i$ belong to a
set $I\subset J$ that satisfies the DCC, then we would be done.  In fact, if
$\mathbf{L}_{\Phi}\leq \mathbf{B}$, then it is relatively straightforward to conclude that
the volume satisfies the DCC, see the proof of \eqref{p_limit}.  Unfortunately, since
$\map X_i.Z.$ might extract arbitrarily many divisors, it is all too easy to write down
examples where the smallest set $J$ which contains the coefficients of every $\Phi_i$ does
not satisfy the DCC (cf. \eqref{e-ndcc}).

Therefore, our objective is to find a model $\map Z'.Z.$ and suitable modifications
$\Delta'_i$ of $\Delta_i$ such that $\mathbf{L}_{\Phi'}\leq \mathbf{B}'$, where
$\mathbf{B}'$ is the limit of $\mathbf{M}_{\Delta'_i}$.  We choose $\Delta'_i$ so that the
difference $\Delta_i-\Delta'_i$ is supported only on the strict transform of the
exceptional divisors of $\map Z'.Z.$.  In this case, it is not hard to arrange for the
coefficients of $\Delta'_i$ to belong to a set $I\subset J$ that satisfies the DCC.

We construct $\map Z'.Z.$ by induction.  For this, we can work locally about a point $p$
in $Z$.  Suppose that $p$ is the intersection of two components $B_1$ and $B_2$ of $\Phi$.
If $B_1$ and $B_2$ appear with coefficient $b_1$ and $b_2$ in $\Phi$ and $\pi\colon \map
Z'.Z.$ blows up $p$, then working locally about $p$, we may write
$$
K_{Z'}+b_1B'_1+b_2B'_2+eE=\pi^*(K_Z+b_1B_1+b_2B_2),
$$
where $e=b_1+b_2-1$.  Here primes denote strict transforms and $E$ is the unique
exceptional divisor.  We suppose that $\mathbf{L}_{\Phi}\leq \mathbf{B}$ does not hold, so
that there is some valuation $\nu$, with centre $p$, such that
$\mathbf{L}_{\Phi}(\nu)>\mathbf{B}(\nu)$.  Since $e=b_1+b_2-1$, the larger $b_1$ and
$b_2$, the further we expect to be from the inequality $\mathbf{L}_{\Phi}\leq \mathbf{B}$.
We introduce the \textbf{weight} $w$, which counts the number of components of $\Phi$ of
coefficient one, that is, the number of $i$ such that $b_i=1$.  In the case of a surface,
the weight is $0$, $1$ or $2$ and it clearly suffices to construct $\map Z'.Z.$ so that
the weight goes down.

In fact, the extreme cases are relatively easy.  If the weight is two, then we just take
$\map Z'.Z.$ to be the blow up of $p$.  The key point is that then the base locus of the
linear system
$$
f'_{i*}|m(K_{X_i}+\Delta_i)|\subset |m(K_Z+\Phi_i)|,
$$
contains $p$ in its support for all $m$ sufficiently large and divisible and this forces
the strict transform of $E$ to be a component of $E_i^+$ as well.  Therefore we are free
to decrease the coefficient of $E$ in $\Phi'$ away from $1$.  At the other extreme, if the
weight is zero, then $(Z,\Phi)$ is kawamata log terminal and we may find $\map Z'.Z.$
which extracts every divisor of coefficient greater than zero.  In this case
$\mathbf{L}_{\Phi'}(\rho)=0$ for every valuation $\rho$ whose centre on $Z'$ is not a
divisor and the inequality $\mathbf{L}_{\Phi'}\leq \mathbf{B}'$ is trivial.

The hard case is when the weight is one, so that one of $b_1$ and $b_2$ is one.  Suppose
that $b_2=1$.  If $\nu$ is a valuation such that $\mathbf{L}_{\Phi}(\nu)>0$ then $\nu$
corresponds to a weighted blow up.  At this point it is convenient to use the language of
toric geometry.  A weighted blow up corresponds to a pair of natural numbers $(v_1,v_2)\in
\mathbb{N}^2$.  In fact, if
$$
\mathfrak{F}=\{\,v_1\in \mathbb{N}  \,|\, (1-b_1)v_1<1 \,\},
$$
then 
$$
\{\,(v_1,v_2)\in \mathbb{N}^2  \,|\, v_1\in \mathfrak{F} \,\},
$$
is the set of all valuations $\nu$ such that $\mathbf{L}_{\Phi}(\nu)>0$.  The crucial
point is that $\mathfrak{F}$ is finite.  For every element $f_1\in \mathfrak{F}$, we pick
$v_2\in \mathbb{N}$, which minimises $\mathbf{B}(f_1,v_2)$ (this makes sense as $I$
satisfies the DCC).  We then pick any simple normal crossings model $\map Z'.Z.$ on which
the centre of every one of these finitely many valuations is a divisor.  Using standard
toric geometry, one can check that on $Z'$ every valuation $\nu'$, whose centre belongs to
a component of $\Phi'$ of coefficient one, satisfies the property
$\mathbf{L}_{\Phi'}(\nu')\leq \mathbf{B}'(\nu')$.  It follows that the weight of
$(Z',\Phi')$ is zero and this completes the induction.  The details are contained in the
proof of \eqref{p_limit}.

\begin{example}\label{e-ndcc} Consider the behaviour of $\mathbf{L}_{\Phi}$ in a simple example.  
We use the notation above.  Note that $e=b_1+b_2-1$ is an increasing and affine linear
function of $b_1$ and $b_2$, whence a continuous function of $b_1$ and $b_2$.  We are only
concerned with the possibility that $e>0$, and in this case,
$\mathbf{L}_{\Phi,Z'}=b_1B'_1+b_2B'_2+eE$, by definition.  There are two interesting
points $p_1=B_1'\cap E$ and $p_2=B_2'\cap E$ lying over $p$, and the coefficients of the
divisors containing them are $b_1$, $e$ and $b_2$, $e$.  The problem is that we can blow
up along either point and keep going.

Suppose that
$$
I=\{\, \frac{r-1}r \,|\, r\in \mathbb{N} \,\}.
$$
Note that if $b_1=\frac ir$ and $b_2=\frac{r-1}r$ then 
$$
e=b_1+b_2-1=\frac ir+\frac{r-1}r-1=\frac{i-1}r.
$$
So, the smallest set $J$ which contains $I$ and which is closed under the operation of
picking any two elements $b_1$ and $b_2$ and replacing them by $b_1+b_2-1$ (provided 
this sum is non-negative) is $\mathbb{Q}\cap [0,1]$.  Clearly, this set does not satisfy the DCC.
\end{example}

\section{Preliminaries}

\makeatletter
\renewcommand{\thetheorem}{\thesubsection.\arabic{theorem}}
\@addtoreset{theorem}{subsection}
\makeatother

\subsection{Notation and Conventions}\label{ss-NC} We will use the notations in
\cite{KM98} and \cite{Lazarsfeld04b}.

If $D=\sum d_{i}D_{i}$ is a $\mathbb{Q}$-divisor on a normal variety $X$, then the
\textit{round down} of $D$ is $\rfdown D.=\sum \rfdown d_i.D_i$, where $\rfdown d.$
denotes the largest integer which is at most $d$, the \textit{fractional part} of $D$ is
$\{D\}=D-\rfdown D.$, and the \textit{round up} of $D$ is $\rfup D. =-\rfdown -D.$.  If
$D'=\sum d'_iD_i$ is another $\mathbb{Q}$-divisor, then $D\wedge D':=\sum \min\{d_i,
d'_i\}D_i$.

The sheaf $\ring X.(D)$ is defined by
$$
\ring X.(D)(U)=\{\, f\in k(X) \,|\, (f)|_U+D|_U\geq 0 \,\},
$$
so that $\ring X.(D)=\ring X.(\rfdown D.)$.  Similarly we define $|D|=|\rfdown D.|$.
If $X$ is normal, and $D$ is a $\mathbb{Q}$-divisor on $X$, the \textit{rational map $\phi_D$
associated to $D$} is the rational map determined by the restriction of $\rfdown D.$ to
the smooth locus of $X$. 

If $X$ is a normal projective variety and $D$ is a $\mathbb{Q}$-Cartier divisor,
$\kappa_{\sigma}(X,D)$ denotes the \textit{numerical Kodaira dimension}, which is defined
by Nakayama in \cite[V.2.5]{Nakayama04} as follows: let $H$ be a divisor on $X$, we define
$$
\sigma(X,D;H):=\max \{\, k\in \mathbb{N} \,|\, \limsup_{m\to \infty}\frac{h^0(X,H+mD)}{m^k}>0 \,\},
$$
and
$$
\kappa_{\sigma}(X,D) := \max\{\, \sigma(X,D;H) \,|\, \text{$H$ is a divisor}\,\}.
$$
If $D$ is pseudo-effective, we define $N_{\sigma}(X,D)$ as in \cite[III.4]{Nakayama04} or \cite[3.3.1]{BCHM06}.

A \textit{log pair} $(X,\Delta)$ consists of a normal variety $X$ and a $\mathbb{Q}$-Weil
divisor $\Delta\geq 0$ such that $K_X+\Delta$ is $\mathbb{Q}$-Cartier.  The
\textit{support} of $\Delta=\sum_{i\in I}d_i\Delta_i$ is the sum $D=\sum_{i\in
  I}\Delta_i$.  If $(X,\Delta)$ has simple normal crossings then a \textit{stratum} of
$(X,\Delta)$ is an irreducible component of the intersection $\cap_{j\in J}\Delta_j$,
where $J$ is a non-empty subset of $I$ (in particular, a stratum is always a proper closed
subset of $X$).  If we are given a morphism $\map X.T.$, then we say that $(X,\Delta)$ has
\textit{simple normal crossings over $T$} if $(X,\Delta)$ has simple normal crossings and
both $X$ and every stratum of $(X,\Delta)$ is smooth over $T$.  We say that the birational
morphism $f\colon\map Y.X.$ only \textit{blows up strata} of $(X,\Delta)$, if $f$ is the
composition of birational morphisms $f_i\colon\map X_{i+1}.X_i.$, $1\leq i\leq k$, with
$X=X_0$, $Y=X_{k+1}$, and $f_i$ is the blow up of a stratum of $(X_i,D_i)$, where $D_i$ is
the sum of the strict transform of $D$ and the exceptional locus.

A \textit{log resolution} of the pair $(X,\Delta)$ is a projective birational morphism
$\mu\colon\map Y.X.$ such that the exceptional locus is the support of a $\mu$-ample
divisor and $(Y,G)$ has simple normal crossings, where $G$ is the support of the strict
transform of $\Delta$ and the exceptional divisors.  Note that the extra assumption that
the exceptional locus is the support of a $\mu$-ample divisor is not standard.  However it
is convenient for our purpose, and it can be always achieved after possibly choosing a
higher model.  If we write
$$ 
K_Y+\Gamma+\sum a_iE_i=\mu^*(K_X + \Delta)
$$ 
where $\Gamma$ is the strict transform of $\Delta$, then $a_i$ is called the
\textit{coefficient of $E_i$ with respect to $(X,\Delta)$}.  Note that $-a_i$ is the
discrepancy of the pair $(X,\Delta)$, with respect to $E_i$, see \cite[2.25]{KM98}.  A
\textit{non kawamata log terminal centre} is the centre of any valuation $\nu$ whose
coefficient is at least one.

In this paper, we only consider valuations $\nu$ of $X$ whose centre on some birational
model $Y$ of $X$ is a divisor.  We say that a formal sum $\mathbf{B}=\sum a_{\nu}\nu$,
where the sum ranges over all valuations of $X$, is a \textit{$\mathbf{b}$-divisor}, if
the set
$$
F_X=\{\, \nu \,|\, \text{$a_{\nu}\neq 0$ and the centre of $\nu$ on $X$ is a divisor} \,\},
$$
is finite.  The \textit{trace} $\mathbf{B}_Y$ of $\mathbf{B}$ is the sum $\sum a_{\nu}
B_{\nu}$, where the sum now ranges over the elements of $F_Y$.  In fact, to give a
$\textbf{b}$-divisor is the same as to give a collection of divisors on every birational
model of $X$, which are compatible under pushforward.

\subsection{Log pairs}

\begin{definition}\label{d_global-quotient} We say that a log pair $(X,\Delta)$ is a
\textbf{global quotient} if there is a smooth quasi-projective variety $Y$ and a finite
subgroup $G\subset \Aut(Y)$ such that $X=Y/G$ and if $\pi\colon\map Y.X.$ is the quotient
morphism, then $K_Y=\pi^*(K_X+\Delta)$.
\end{definition}

Note that if $(X,\Delta)$ is a global quotient, then $X$ is $\mathbb{Q}$-factorial,
$(X,\Delta)$ is kawamata log terminal, and the coefficients of $\Delta$ belong to the
set
$$
\{\, \frac{r-1}r \,|\, r\in \mathbb{N}\,\},
$$
(cf. \cite[5.15, 5.20]{KM98}).

\begin{lemma}\label{l_round} If $(X,\Delta)$ is a log pair and the coefficients of
$\Delta$ are less than one, then
$$
\rfdown m\Delta.\leq \rfup (m-1)\Delta.,
$$
for every positive integer $m$, with equality if the coefficients of $\Delta$ belong to
the set $\ds{\{\, \frac{r-1}r \,|\, r\in \mathbb{N} \,\}}$.
\end{lemma}
\begin{proof} Easy.  
\end{proof}

\subsection{The volume}

\begin{definition}\label{d_volume} Let $X$ be a normal $n$-dimensional irreducible
projective variety and let $D$ be a $\mathbb{Q}$-divisor.  The \textbf{volume} of $D$ is
$$
\vol(X,D) = \limsup_{m\to \infty}\frac{n! h^0(X,\ring X.(mD))}{m^n}.
$$
We say that $D$ is \textbf{big} if $\vol(X,D)>0$.
\end{definition}

For more background, see \cite{Lazarsfeld04b}.  We will need the following simple result:
\begin{lemma}\label{l_standard} Let $X$ be a projective variety, $D$ a divisor such that 
the rational map $\phi_D\colon\rmap X.{\pr n.}.$ is birational onto its image $Z$.  Then,
the volume of $D$ is greater than or equal to the degree of $Z$.  In particular the volume
of $D$ is at least $1$.
\end{lemma}
\begin{proof} This is well known, see for example (2.2) of \cite{HM05b}.
\end{proof}

\begin{definition}\label{d_potential} Let $X$ be a normal projective variety and let
$D$ be a big $\mathbb{Q}$-Cartier $\mathbb{Q}$-divisor on $X$.

If $x$ and $y$ are two very general points of $X$ then, possibly switching $x$ and $y$, we
may find $0\leq \Delta\sim_{\mathbb{Q}}(1-\epsilon) D$, for some $0<\epsilon <1$, where
$(X,\Delta)$ is not kawamata log terminal at $y$, $(X,\Delta)$ is log canonical at $x$ and
$\{x\}$ is a non kawamata log terminal centre, then we say that $D$ is \textbf{potentially
  birational}.
\end{definition}

\begin{lemma}\label{l_potential} Let $X$ be a normal projective variety and let
$D$ be a big $\mathbb{Q}$-Cartier divisor on $X$.

\begin{enumerate}
\item If $D$ is potentially birational, then $\phi_{K_X+\rfup D.}$ is birational.
\item If $X$ has dimension $n$ and $\phi_D$ is birational, then $(2n+1)\rfdown D.$ is
potentially birational.  In particular, $\phi_{K_X+(2n+1)D}$ is birational and
$K_X+(2n+1)D$ is big.
\end{enumerate}
\end{lemma}
\begin{proof} Replacing $X$ by a resolution, we may assume that $X$ is smooth.  As $D$ is
big, we may write $D\sim_\mathbb{Q} A+B$ where $A$ is an ample $\mathbb{Q}$-divisor and
$B\geq 0$.  Using $A$ to tie break (cf. \cite[6.9]{Kollar95}), we may assume that
$(X,\Delta)$ is kawamata log terminal in a punctured neighbourhood of $x$.  As
$$
\rfup D.-(\Delta+\epsilon B+\rfup D.-D) \sim_\mathbb{Q} \epsilon A,
$$
is ample, Nadel vanishing implies that
$$
H^1(X,\ring X.(K_X+\rfup D.)\otimes \mathcal{J}(\Delta+\epsilon B+\rfup D.-D))=0,
$$
where $\mathcal{J}(\Delta+\epsilon B+\rfup D.-D)$ is the multiplier ideal sheaf.  But then
we may find a section $\sigma\in H^0(X,\ring X.(K_X+\rfup D.))$ vanishing at $y$ but not
at $x$.  In particular, as $y$ is very general, we may also find $\tau$ not vanishing at
$y$.  But then some linear combination $\rho$ of $\tau$ and $\sigma$ is a section which
vanishes at $x$ and not at $y$.  This is (1).

Replacing $D$ by $\rfdown D.$, we may assume that $D$ is Cartier.  Let $X'$ be the image
of $\phi_D$.  Let $x'=\phi_D(x)$ and $y'=\phi_D(y)$ and let $\Delta'$ be the sum of $n$
general hyperplanes through $x'$ and $n$ general hyperplanes through $y'$.  Let $\Delta$
be the strict transform of $\Delta'$.  As $x$ and $y$ are general, $\phi_D$ is an
isomorphism in a neighbourhood of $x$ and $y$.  It follows that $(X,\Delta)$ is not
kawamata log terminal at $y$, $(X,\Delta)$ is log canonical at $x$ and if we blow up $x$
then the coefficient of the exceptional divisor is one.  It is then easy to see that
$(2n+1)D$ is potentially birational and (2) follows from (1).
\end{proof}

We will need the following result from \cite{Kollar95}:
\begin{theorem}\label{t_potential} Let $(X,\Delta)$ be a kawamata log terminal pair,  where
$X$ is projective.  Suppose that $x$ and $y$ are two closed points of $X$.  Let
$\Delta_0\geq 0$ be a $\mathbb{Q}$-Cartier divisor on $X$ such that $(X,\Delta+\Delta_0)$
is log canonical in a neighbourhood of $x$ but not kawamata log terminal at $y$, and there
is a non kawamata log terminal centre $V$ which contains $x$.  Let $H$ be an ample
$\mathbb{Q}$-divisor on $X$ such that $\vol(V,H|_V)>2 k^k$, where $k=\dim V$.

Then, possibly switching $x$ and $y$, there is a $\mathbb{Q}$-divisor $H\sim_{\mathbb{Q}}
\Delta_1\geq 0$ and rational numbers $0<a_i\leq 1$ such that
$(X,\Delta+a_0\Delta_0+a_1\Delta_1)$ is log canonical in a neighbourhood of $x$ but not
kawamata log terminal at $y$, and there is a non kawamata log terminal centre $V'$ which
contains $x$, such that $\dim V'<k$.
\end{theorem}
\begin{proof} By (6.9.1) of \cite{Kollar95}, we may assume that $V$ is the unique non
kawamata log terminal centre which contains $x$, and we may apply (6.8.1), (6.8.1.3) and
(6.5) of \cite{Kollar95}.
\end{proof}

\begin{theorem}\label{t_recursive} Let $(X,\Delta)$ be a kawamata log terminal pair, where
$X$ is projective of dimension $n$ and let $H$ be an ample $\mathbb{Q}$-divisor.  Suppose
$\gamma_0\geq 1$ is a constant such that $\vol(X,\gamma_0H)>n^n$.  Suppose $\epsilon$ is a
constant with the following property:

For very general $x$ in $X$ and every $0\leq \Delta_0 \sim_{\mathbb{Q}}\lambda H$ such
that $(X,\Delta+\Delta_0)$ is log canonical at $x$, if $V$ is the minimal non kawamata log
terminal centre containing $x$, then $\vol(V,\lambda H|_V)>\epsilon^k$, where $k$ is the
dimension of $V$ and $\lambda\geq 1$ is a rational number.

Then $mH$ is potentially birational, where
$$
m=2\gamma_0\left (1+\gamma\right)^{n-1} \qquad \text{and} \qquad \gamma=\frac {2n}{\epsilon}.
$$
\end{theorem}
\begin{proof} Let $x$ and $y$ be two very general points of $X$.  Possibly switching $x$
and $y$, we will prove by descending induction on $k$ that there is a $\mathbb{Q}$-divisor
$\Delta_0\geq 0$ such that:

\noindent ($\flat$)$_k$ $\Delta_0\sim_{\mathbb{Q}} \lambda H$, for some
$1\leq \lambda<2\gamma_0(1+\gamma)^{n-1-k}$, where $(X,\Delta+\Delta_0)$ is log canonical
at $x$, not kawamata log terminal at $y$ and there is a non kawamata log terminal centre
$V$ of dimension at most $k$ containing $x$.

As
$$
\vol(X,2\gamma_0 H)> 2n^n,
$$
we may find $0\leq \Phi \sim_{\mathbb{Q}} 2\gamma_0 H$, such that
$(X,\Delta+\Phi)$ is not log canonical at either $x$ or $y$.  If
$$
\beta=\sup \{\, \alpha \,|\, \text{$K_X+\Delta+\alpha\Phi$ is log canonical at $x$} \,\},
$$
is the log canonical threshold, then $\beta<1$.  Possibly switching $x$ and $y$, we may
assume that $(X,\Delta+\beta\Phi)$ is not kawamata log terminal at $y$.  Clearly
$\Delta_0=\beta\Phi$ satisfies ($\flat$)$_{n-1}$, so this is the start of the induction.

Now suppose that we may find a $\mathbb{Q}$-divisor $\Delta_0$ satisfying ($\flat$)$_k$.
We may assume that $V$ is the minimal non kawamata log terminal centre containing $x$ and
that $V$ has dimension $k$.  By assumption,
$$
\vol(V,\lambda\gamma H|_V)>2k^k,
$$
so that by \eqref{t_potential}, possibly switching $x$ and $y$, we may find
$\Delta_1\sim_{\mathbb{Q}} \mu H$, where $\mu<\lambda\gamma$ and constants $0<a_i\leq 1$
such that $(X,\Delta+a_0\Delta_0+a_1\Delta_1)$ is log canonical at $x$, not kawamata log
terminal at $y$ and there is a non kawamata log terminal centre $V'$ containing $x$, whose
dimension is less than $k$.  As
$$
a_0\Delta_0+a_1\Delta_1 \sim_{\mathbb{Q}} (a_0\lambda+a_1\mu)H,
$$
and 
$$
\lambda'=a_0\lambda+a_1\mu\leq (1+\gamma)\lambda<2\gamma_0(1+\gamma)^{n-1-(k-1)},
$$
$a_0\Delta_0+a_1\Delta_1+\max(0,1-\lambda')B$ satisfies ($\flat$)$_{k-1}$, where the
support of $B\sim_{\mathbb{Q}} H$ does not contain either $x$ or $y$ (we only need to add
on $B$ in the unlikely event that $\lambda'<1$).  This completes the induction and the
proof.
\end{proof}

\subsection{Bounded pairs}

\begin{definition}\label{d_birationally-bounded} We say that a set $\mathfrak{X}$
of varieties is \textbf{birationally bounded} if there is a projective morphism $\map
Z.T.$, where $T$ is of finite type, such that for every $X\in \mathfrak{X}$, there is a
closed point $t\in T$ and a birational map $f\colon\rmap Z_t.X.$.

We say that a set $\mathfrak{D}$ of log pairs is \textbf{log birationally bounded} if
there is a log pair $(Z,B)$, where the coefficients of $B$ are all one, and a projective
morphism $\map Z.T.$, where $T$ is of finite type, such that for every $(X,\Delta)\in
\mathfrak{D}$, there is a closed point $t\in T$ and a birational map $f\colon\rmap Z_t.X.$
such that the support of $B_t$ contains the support of the strict transform of $\Delta$
and any $f$-exceptional divisor.
\end{definition}

\begin{lemma}\label{l_birationally-bounded} Fix a positive integer $n$.
\begin{enumerate}
\item Let $\mathfrak{X}$ and $\mathfrak{Y}$ be two sets of varieties, such that if $X\in
\mathfrak{X}$, then we may find $Y\in \mathfrak{Y}$ birational to $X$.  If $\mathfrak{Y}$
is birationally bounded, then $\mathfrak{X}$ is birationally bounded.
\item Let $\mathfrak{X}$ be a set of varieties of dimension $n$.  If there is a constant
$V$ such that for every $X\in \mathfrak{X}$, we may find a Weil divisor $D$ such that
$\phi_D$ is birational and the volume of $D$ is at most $V$, then $\mathfrak{X}$
is birationally bounded.
\item Let $\mathfrak{D}$ and $\mathfrak{G}$ be two sets of log pairs, such that if
$(X,\Delta)\in \mathfrak{D}$, then we may find $(Y,\Gamma)\in \mathfrak{G}$, and a
birational map $f\colon\rmap Y.X.$, where the support of $\Gamma$ contains the support of the
strict transform of $\Delta$ and any $f$-exceptional divisor.  If $\mathfrak{G}$ is log
birationally bounded, then $\mathfrak{D}$ is log birationally bounded.
\item Let $\mathfrak{D}$ be a set of log pairs of dimension $n$.  If there are constants
$V_1$ and $V_2$ such that for every $(X,\Delta)\in \mathfrak{D}$ we may find a Weil
divisor $D$ such that $\phi_D\colon\rmap X.Y.$ is birational, the volume of $D$ is at most
$V_1$, and if $G$ denotes the sum over the components of the strict transform of $\Delta$
and the $\phi^{-1}$-exceptional divisors, then $G\cdot H^{n-1}\leq V_2$, where $H$ is the
very ample divisor on $Y$ determined by $\phi_D$, then $\mathfrak{D}$ is log birationally
bounded.
\item If the set $\mathfrak{D}$ of log pairs is log birationally bounded, then
$$
\mathfrak{X}=\{\, X \,|\, (X,\Delta)\in \mathfrak{D} \,\}
$$
is birationally bounded.
\end{enumerate}
\end{lemma}
\begin{proof} (1), (3) and (5) are clear.

We prove (4).  Suppose that $Y\subset \pr s.$ is a closed subvariety of dimension $n$ and
degree at most $V_1$.  Then by the classification of minimal degree subvarieties of
projective space, we may assume $s\leq V_1+1-n$.  By boundedness of the Chow variety,
there are flat morphisms $\map Z.T.$ and $\map B.T.$ such that if $Y\subset \pr s.$ has
dimension $n$ (respectively $n-1$) and degree at most $V_1$ (respectively $V_2$), then $Y$
is isomorphic to the fibre $Z_t$ (respectively $B_t$) over a closed point $t\in T$.
Passing to a stratification of $T$ and a log resolution of the generic fibres of $\map
Z.T.$, we may assume that the fibres of $\map Z.T.$ are smooth.  In particular, $(Z,B)$ is
a log pair.

Now suppose that $(X,\Delta)\in \mathfrak{D}$.  By assumption there is a divisor $D$ such
that $\phi_D\colon\rmap X.Y.$ is birational.  The degree of the image is at most the
volume of $D$, that is, at most $V_1$.  So there is a closed point $t\in T$ such that $Y$
is isomorphic to $Z_t$.  By assumption $G\cdot H^{n-1}\leq V_2$ so that we may assume that
$G$ corresponds to $B_t$.  But then $\mathfrak{D}$ is log birationally bounded.  This is
(4).

The proof of (2) is similar to and easier than the proof of (4).
\end{proof}

\makeatletter
\renewcommand{\thetheorem}{\thesection.\arabic{theorem}}
\@addtoreset{theorem}{section}
\makeatother

\section{Birationally bounded pairs}
\label{s_log}

\S \ref{s_log} is devoted to a proof of:
\begin{theorem}\label{t_boundingimage} Fix a positive integer $n$ and two constants $A$
and $\delta>0$.

Then the set of log pairs $(X,\Delta)$ satisfying
\begin{enumerate}
\item $X$ is projective of dimension $n$,
\item $(X,\Delta)$ is log canonical,
\item the coefficients of $\Delta$ are at least $\delta$, 
\item there is a positive integer $m$ such that $\vol(X,m(K_X+\Delta))\leq A$, and
\item $\phi_{K_X+m(K_X+\Delta)}$ is birational, 
\end{enumerate}
is log birationally bounded.
\end{theorem}

The key result is:
\begin{lemma}\label{l_compare-volumes} Let $X$ be a normal projective variety of dimension
$n$ and let $M$ be a base point free Cartier divisor such that $\phi_M$ is birational.
Let $H=2(2n+1)M$.

If $D$ is a sum of distinct prime divisors, then
$$
D\cdot H^{n-1}\leq 2^n\vol(X,K_X+D+H).
$$
\end{lemma}
\begin{proof} Possibly discarding $\phi_M$-exceptional components of $D$, we may assume
that no component of $D$ is $\phi_M$-exceptional.  If $f\colon\map Y.X.$ is a log
resolution of the pair $(X,D)$ and $G$ is the strict transform of $D$, then
$$
D\cdot H^{n-1}=G\cdot (f^*H)^{n-1},
$$
and
$$
\vol(Y,K_Y+G+f^*H)\leq \vol(X,K_X+D+H).
$$
Replacing $(X,D)$ by $(Y,G)$ and $M$ by $f^*M$ we may assume that $(X,D)$ has simple
normal crossings, and possibly blowing up more, that the components of $D$ do not
intersect.

Since no component of $D$ is contracted, we may find an ample $\mathbb{Q}$-divisor $A$ and
a $\mathbb{Q}$-divisor $B\geq 0$, such that
$$
M\sim_{\mathbb{Q}} A+B,
$$
where $B$ and $D$ have no common components.  As $K_X+D+\delta B$ is divisorially log
terminal for any $\delta>0$ sufficiently small, it follows that
$$
H^i(X,\ring X.(K_X+E+pM))=0,
$$
for all positive integers $p$, $i>0$ and any integral Weil divisor $0\leq E\leq D$.
If we let
$$
A_m=K_X+D+mH,
$$
then
$$
H^i(D,\ring D.(A_m))=0,
$$
for all $i>0$ and $m>0$ and so there is a polynomial $P(m)$ of degree $n-1$, with
$$
P(m)=h^0(D,\ring D.(A_m)),
$$
for $m>0$.  (2) of \eqref{l_potential} implies that $A_1=K_X+D+H$ is big and so
\cite{BCHM06} implies that $K_X+D+H$ has a log canonical model.  In particular there is a
polynomial of $Q(m)$ of degree $n$, with
$$
Q(m)=h^0(X,\ring X.(2mA_1)),
$$
for any sufficiently divisible positive integer $m$.  Note that the leading coefficients
of $P(m)$ and $Q(m)$ are
$$
\frac{D\cdot H^{n-1}}{(n-1)!} \qquad \text{and} \qquad \frac{2^n\vol(X,K_X+D+H)}{n!}.
$$
If $D_i$ is a component of $D$, and $M_i=(D-D_i+(2n+1)M)|_{D_i}$, then
$$
\map {H^0(X,\ring X.(K_X+D+(2n+1)M))}.{H^0(D_i,\ring D_i.(K_{D_i}+M_i))}.,
$$
is surjective, and so the general section of $H^0(X,\ring X.(K_X+D+(2n+1)M))$ does not
vanish identically on any component of $D$.  Pick sections
$$
s\in H^0(X,\ring X.(K_X+D+(2n+1)M)) \quad \text{and} \quad l\in H^0(X,\ring X.((2n+1)M)),
$$
whose restrictions to each component of $D$ is non-zero.  Let
$$
t=s^{\otimes 2m-1}\otimes l\in H^0(X,\ring X.(2mA_1-A_m)).
$$
Consider the following commutative diagram
\begin{diagram}
0&\rTo & \ring X.(A_m-D)&\rTo& \ring X.(A_m) &\rTo &\ring D.(A_m) & \rTo& 0\\
 & & \dTo& &\dTo& &\dTo\\
0&\rTo & \ring X.(2mA_1-D)&\rTo& \ring X.(2mA_1)&\rTo & \ring D.(2mA_1) &\rTo& 0,
\end{diagram}
where the vertical morphisms are injections induced by multiplying by $t$.

Note that
$$
\map {H^0(X,\ring X.(A_m))}.{H^0(D,\ring D.(A_m))}.,
$$
is surjective.  Hence every element of $H^0(D,\ring D.(2mA_1))$ in the image of
$H^0(D,\ring D.(A_m))$ lifts to $H^0(X,\ring X.(2mA_1))$.  Therefore
$$
P(m)\leq h^0(X,\ring X.(2mA_1))-h^0(X,\ring X.(2mA_1-D)).
$$
Note that
$$
Q(m-1)=h^0(X,\ring X.(2(m-1)A_1))\leq h^0(X,\ring X.(2mA_1-D)),
$$
as $h^0(X,\ring X.(2K_X+D+2H))\neq 0$.  It follows that
$$
P(m)\leq Q(m)-Q(m-1).
$$
Now compare the leading coefficients of $P(m)$ and $Q(m)-Q(m-1)$.
\end{proof}

\begin{proof}[Proof of \eqref{t_boundingimage}] Let $(X,\Delta)$ be a log pair satisfying
the hypotheses of \eqref{t_boundingimage}.  If $\pi\colon\map Y.X.$ is a log resolution of
$(X,\Delta)$ which resolves the indeterminacy of 
$$
\phi=\phi_{K_X+m(K_X+\Delta)}\colon\rmap X.Z.,
$$
and $\Gamma$ is the strict transform of $\Delta$ plus the sum of the exceptional divisors,
then $(X,\Delta)$ is log birationally bounded if and only if $(Y,\Gamma)$ is log
birationally bounded, by (3) of \eqref{l_birationally-bounded}.  On the other hand,
$$
\vol(Y,m(K_Y+\Gamma))\leq \vol(X,m(K_X+\Delta))\leq A,
$$
and $\phi_{K_Y+m(K_Y+\Gamma)}$ is birational.  

Replacing $(X,\Delta)$ by $(Y,\Gamma)$, we may assume that
$$
\phi=\phi_{K_X+m(K_X+\Delta)}\colon\map X.Z.,
$$
is a birational morphism.  In particular, if we decompose $\rfdown K_X+m(K_X+\Delta).$
into its mobile part $M$ and its fixed part $E$, so that
$$
|K_X+m(K_X+\Delta)|=|M|+E,
$$
then $M$ is big and base point free.  Let $H$ be a divisor on $Z$ such that $M=\phi^*H$,
so that $H$ is very ample.

Note that
\begin{align*} 
\vol(X,K_X+m(K_X+\Delta))&\leq \vol(X,(m+1)(K_X+\Delta))\\
                         &\leq 2^nA.
\end{align*} 

On the other hand, let $G$ be the sum of the components of the strict transform of
$\Delta$ on $Z$.  Pick $B\in |\rfdown K_X+m(K_X+\Delta).|$.  Let
$$
\alpha=\max(\frac 1{\delta},2(2n+1)).
$$
If $D_0$ is the sum of the components of $\Delta$ and $B$ which are not contracted by
$\phi$, then 
$$
D_0\leq \alpha(B+\Delta).
$$
Note that there is a divisor $C\geq 0$ such that 
$$
\alpha(B+\Delta)+C\sim_{\mathbb{Q}} \alpha(m+1)(K_X+\Delta).
$$
As $\phi$ is a morphism and $M$ is base point free, \eqref{l_compare-volumes} implies that
\begin{align*}
G\cdot H^{n-1}&\leq D_0\cdot (2(2n+1)M)^{n-1} \\
             &\leq 2^n\vol(X,K_X+D_0+2(2n+1)M) \\
             &\leq 2^n\vol(X,K_X+\alpha(B+\Delta)+2(2n+1)(M+E+\Delta))\\
             &\leq 2^n\vol(X,K_X+\Delta+\alpha(m+1)(K_X+\Delta)+2(2n+1)(m+1)(K_X+\Delta))\\
             &\leq 2^n(1+2\alpha(m+1))^n\vol(X,K_X+\Delta)\\
             &\leq 2^{3n}\alpha^n\vol(X,(m+1)(K_X+\Delta))\\
             &\leq 2^{4n}\alpha^nA.
\end{align*}

Now apply (4) of \eqref{l_birationally-bounded}. \end{proof}

\section{Deformation invariance of log plurigenera}
\label{s_deformation}

\begin{proposition}\label{p_mmp} Let $(X,\Delta)$ be a $\mathbb{Q}$-factorial log
pair.  Suppose that $\map X.T.$ is a projective morphism to a smooth curve $T$ whose
fibres $(X_t,\Delta_t)$ are terminal, where every component of $\Delta$ dominates $T$.
Let $0\in T$ be a closed point.  Suppose that
\begin{itemize} 
\item either $\Delta$ or $K_X+\Delta$ is big over $T$, and 
\item no component of $\Delta_0$ belongs to the stable base locus of $K_{X_0}+\Delta_0$.
\end{itemize} 

Then we may find a log terminal model $f\colon\rmap X.Y.$ of $(X,\Delta)$ over $T$, such
that $f$ is an isomorphism at the generic point of every component of $\Delta_0$ and
$f_0\colon\rmap X_0.Y_0.$ is a weak log canonical model of $(X_0,\Delta_0)$.
\end{proposition}
\begin{proof} We first prove this result under the additional hypothesis that $K_X+\Delta$
is pseudo-effective over $T$.

Note that $X$ is smooth in codimension two, as the fibres $X_t$ of $\map X.T.$ are Cartier
and smooth in codimension two.  Hence
$$
(K_X+\Delta)|_{X_0}=(K_X+X_0+\Delta)|_{X_0}=K_{X_0}+\Delta_0.
$$
It follows from \cite[1.4.5]{BCHM06} that the coefficient of any valuation $\mu$ with
respect to $(X,X_0+\Delta)$ is at most zero, if the centre of $\mu$ is neither a component
of $\Delta$ nor a component of $\Delta_0$.

By \cite[1.2]{BCHM06}, we may run the $(K_X+\Delta)$-MMP over $T$, that is, we may find a
sequence $\ulist g.m-1.$ of divisorial contractions and flips $g^k\colon\rmap
X^k.X^{k+1}.$ starting at $X=X^1$ and ending with a log terminal model $Y=X^m$ for the
pair $(X,\Delta)$ over $T$.  Let $\Delta^k$ denote the pushforward of $\Delta$ under the
induced birational map $f^k\colon\rmap X.X^k.$.  We will prove by induction on $k$ that
\begin{enumerate}[(a)]
\item $g^k$ is an isomorphism at the generic point of any component of $\Delta_0^k$, and
\item $g^k_0\colon\rmap X_0^k.X_0^{k+1}.$ is a birational contraction.
\end{enumerate} 
Suppose that (a--b)$_{\leq k-1}$ hold.  Then $f^k_0\colon\rmap X_0.X_0^{k}.$ is a
birational contraction which does not contract any components of $\Delta_0$ and so
$(X^k_0,\Delta_0^k)$ is terminal.

(b)$_{\leq k-1}$ implies that no component of $\Delta_0^k$ is a component of the stable
base locus of $K_{X^k_0}+\Delta^k_0$.  Suppose that $g^k$ is not an isomorphism at the
generic point of a divisor $D$ contained in $X_0^k$.  Then $D$ is covered by curves $C$
such that
$$
(K_{X_0^k}+\Delta_0^k)\cdot C=(K_{X^k}+\Delta^k)\cdot C<0.
$$
It follows that $D$ is a component of the stable base locus of $K_{X^k_0}+\Delta^k_0$, so
that $D$ is not a component of $\Delta^k_0$.  Thus (a)$_k$ holds.

Suppose that $G\subset X^{k+1}_0$ is a prime divisor, which is not a component of
$\Delta^{k+1}_0$.  By the classification of log canonical surface singularities, we may
find a valuation $\nu$ with centre $G$ on $X^{k+1}$ whose coefficient $d$ with respect to
$(X^{k+1},X_0^{k+1}+\Delta^{k+1})$ is at least zero.  As $X_0^k$ is the pullback of a
divisor from $T$, $g^k$ is $(K_{X^k}+X_0^k+\Delta^k)$-negative and so the coefficient $c$
of $\nu$ with respect to $(X^k,X_0^k+\Delta^k)$ is at least $d$, with equality if and only
if $g^k$ is an isomorphism at the generic point of $G$.  By (a)$_k$ the centre of $\nu$ on
$X^k$ is not a component of $\Delta_0^k$.  It follows that $0\leq d\leq c\leq 0$, so that
$c=d=0$ and $g^k$ is an isomorphism at the generic point of $G$.  Hence $g^k_0$ is a
birational contraction, that is, (b)$_k$ holds.  This completes the induction and the
proof that (a--b)$_{\leq m-1}$ hold.

As $g^k_0$ is a birational contraction which is $(K_{X_0^k}+\Delta_0^k)$-negative, for
$k\leq m-1$, it follows that $f_0$ is a $(K_{X_0}+\Delta_0)$-negative birational
contraction.  But then $f_0$ is a weak log canonical model.

It remains to prove that $K_X+\Delta$ is pseudo-effective over $T$.  Pick a divisor $A$
which is ample over $T$ and let
$$
\lambda=\inf \{\, t\in \mathbb{R} \,|\, \text{$K_X+\Delta+tA$ is $\pi$-pseudo-effective} \,\},
$$
be the $\pi$-pseudo-effective threshold.  It is proved in \cite{BCHM06} that $\lambda$ is
rational.  By what we have already proved we may find a log terminal model $f\colon\rmap
X.Y.$ of $K_X+\Delta+\lambda A$ over $T$ such that $f_0\colon\rmap X_0.Y_0.$ is a weak log
canonical model of $K_{X_0}+\Delta_0+\lambda A_0$.  Let $G=f_*(K_X+\Delta+\lambda A)$.  If
$\lambda>0$ then $K_{X_0}+\Delta_0+\lambda A_0$ is big, so that $G_0$ is big and nef.  But
then $G_0^n>0$ is positive and so $G_t^n=G_0^n>0$ for every $t\in T$.  As $G$ is nef over
$T$, it is big over $T$, and so $K_X+\Delta+\lambda A$ is big over $T$, a contradiction.
It follows that $\lambda=0$ so that $K_X+\Delta$ is pseudo-effective over $T$.
\end{proof}

\begin{theorem}\label{t_inv} Let $\map X.T.$ be a flat projective morphism of quasi-projective 
varieties.  Let $(X,\Delta)$ be a log pair such that the fibres $(X_t,\Delta_t)$ are
$\mathbb{Q}$-factorial terminal, for every $t\in T$.  Assume that every component $R$ of
$\Delta$ dominates $T$ and that the fibres of the Stein factorisation of $\map R.T.$ are
irreducible.  Let $m>1$ be any integer such that $D=m(K_X+\Delta)$ is integral.

If either $K_X+\Delta$ or $\Delta$ is big over $T$ then $h^0(X_t,\ring X_t.(D_t))$ is
independent of $t\in T$.
\end{theorem}
\begin{proof} Let $R$ be a component of $\Delta$, let $\map S.T.$ be the normalisation of
the Stein factorisation of $\map R.T.$ so that $\map S.T.$ is finite and $S$ is normal,
and let
$$
\begin{diagram}
Y   &   \rTo    &  X  \\
\dTo &  & \dTo\\
S   &   \rTo    & T,
\end{diagram}
$$
be the fibre square.  As $\map S.T.$ is finite, $S$ is irreducible and $\map Y.S.$ is
flat, $Y$ is a quasi-projective variety.  $Y$ is normal by \cite[5.12.7]{EGAIV}.
Replacing $\map X.T.$ by $\map Y.S.$ finitely many times, we may assume that the fibres of
$\map R.T.$ are irreducible, for every component $R$ of $\Delta$.  Fix a closed point
$0\in T$.  Replacing $T$ by the intersection of general hyperplane sections containing
$0$, we may assume that $T$ is a curve.  Replacing $T$ by its normalisation and passing to
the fibre square, we may assume that $T$ is smooth.  As the fibres of $\map X.T.$ are
$\mathbb{Q}$-factorial terminal, \cite[3.2]{FH11} implies that $X$ is
$\mathbb{Q}$-factorial.

It suffices to show that $|D_0|=|D|_{X_0}$.  In particular we may suppose that
$K_{X_0}+\Delta_0$ is pseudo-effective.  Further we are free to work locally about $0$.
In particular we may assume that $T$ is affine.  \cite[1.2]{BCHM06} implies that the
divisor $N_{\sigma}(X_0,K_{X_0}+\Delta_0)$, defined in \eqref{ss-NC}, is a
$\mathbb{Q}$-divisor.  In particular
$$
\Theta_0=\Delta_0-\Delta_0\wedge N_{\sigma}(X_0,K_{X_0}+\Delta_0),
$$ 
is a $\mathbb{Q}$-divisor.  By assumption we may find a $\mathbb{Q}$-divisor $0\leq
\Theta\leq \Delta$ whose restriction to $X_0$ is $\Theta_0$.  If we set
$$
\mu=\frac m{m-1},
$$
then $K_X+\mu\Delta$ is big.  Therefore, we may find $\mathbb{Q}$-divisors $A\geq 0$ and
$B\geq 0$, where $A$ is ample, the support of $A$ is a prime divisor and $X_0$ is not a
component of $B$, such that $K_X+\mu\Delta \sim_{\mathbb{Q}} A+B$.  Possibly passing to an
open subset of $T$ we may assume that the components of $B$ dominate $T$.  Pick $\delta\in
(0,1/2)$ such that $(X_t,\Delta_t+\delta (A_t+B_t))$ is terminal for every $t\in T$.  If
we let
$$
H=\frac {\delta}{m-1-\delta} A,
$$
then 
$$
D-\Xi \sim_{\mathbb{Q}} K_X+\Phi+\delta B+(m-1-\delta)(K_X+\Theta+H),
$$
where $\Phi=(1-\delta\mu+\delta)\Delta$ and $\Xi=(m-1-\delta)(\Delta-\Theta)$.  

As $H_0$ is ample, no component of $\Theta_0+H_0$ belongs to the stable base locus of
$K_{X_0}+\Theta_0+H_0$.  \eqref{p_mmp} implies that we may find a log terminal model
$f\colon\rmap X.Y.$ for $(X,\Theta+H)$, such that the induced birational map
$f_0\colon\rmap X_0.Y_0.$ is a weak log canonical model of $(X_0,\Theta_0+H_0)$.

Let $p\colon\map W.X.$ and $q\colon\map W.Y.$ resolve $f\colon\rmap X.Y.$, where $p$ is
also a log resolution of $(X,\Delta+A+B)$.  If we let 
$$
G=(m-1-\delta)f_*(K_X+\Theta+H),
$$
then $G$ is big and nef and
$$
(m-1-\delta)p^*(K_X+\Theta+H)=q^*G+F,
$$
where $F\geq 0$ is exceptional for $q$.  

Let $W_0$ be the strict transform of $X_0$.  As $(X_0,\Phi_0+\delta B_0)$ is kawamata log
terminal, inversion of adjunction, \cite[1.4.5]{BCHM06}, implies that $(X,X_0+\Phi+\delta
B)$ is purely log terminal.  Therefore, if we write
$$
K_W+W_0=p^*(K_X+X_0+\Phi+\delta B)+E,
$$
then $\rfup E.\geq 0$ is exceptional for $p$.  Let
$$
L=\rfup p^*(D-\Xi)+E-F..
$$
Possibly passing to an open subset of $T$, we may assume $X_0$ is $\mathbb{Q}$-linearly
equivalent to zero.  In particular, 
$$
L-W_0 \sim_{\mathbb{Q}} K_W+C+q^*G,
$$ 
where $C$ is the fractional part of $-p^*(D-\Xi)-E+F$.  Hence $(W,C)$ is kawamata log
terminal and Kawamata-Viehweg vanishing implies that
$$
H^1(W,\ring W.(L-W_0))=0.
$$

Let $N=p^*(K_X+\Theta)-q^*f_*(K_X+\Theta)$.  As $H$ is ample, $p^*H\leq q^*f_*H$, and so
$$
mN=(1+\delta)N+(m-1-\delta)N\geq F.
$$  
As $\Xi\leq m(\Delta-\Theta)$ we have $D-\Xi\geq m(K_X+\Theta)$ and so it follows that
\begin{align*} 
M&=L-\rfdown mq^*f_*(K_X+\Theta).\\
 &=\rfup L-mq^*f_*(K_X+\Theta).\\
 &\geq \rfup mN+E-F.\\
 &\geq \rfup E..
\end{align*} 
Let $q_0\colon\map W_0.X_0.$ denote the restriction of $q$ to $W_0$ and let $L_0$ and
$M_0$ denote the restriction of $L$ and $M$ to $W_0$.  We have
\begin{align*} 
|D_0|&=|m(K_{X_0}+\Theta_0)| && \text{by definition of $\Theta_0$} \\
     &\subset |m f_{0*}(K_{X_0}+\Theta_0)| && \text{since $f_0$ is a birational contraction} \\
     &=|m q_0^*f_{0*}(K_{X_0}+\Theta_0)| && \\
     &\subset |L_0| && \text{as $M_0\geq 0$} \\
     &=|L|_{W_0} && \text{since $H^1(W,\ring W.(L-W_0))=0$}\\          
     &\subset |D|_{X_0} && \text{since $\rfup E.$ is exceptional for $p$.}
\end{align*}  
Thus equality holds, as the reverse inequality holds automatically.
\end{proof}

\begin{proof}[Proof of \eqref{t_dilp}] We first prove (1).  Let $0\in T$ be a closed
point.  Replacing $T$ by an unramified cover, we may assume that the strata of
$(X,\Delta)$ intersect $X_0$ in strata of $(X_0,\Delta_0)$.  Since the only valuations of
non-negative coefficient lie over the strata of $(X,\Delta)$, replacing $(X,\Delta)$ by a
blow up, we may assume that $(X,\Delta)$ is terminal.  Thus (1) follows from
\eqref{t_inv}.

Now we prove (2).  Pick $m_0>0$ such that $m_0\Delta$ is integral and an ample divisor $H$
such that $H+m_0\Delta$ is very ample.  Pick a prime divisor $A\sim H+m_0\Delta$, such
that $(X,\Delta+A)$ has simple normal crossings over $T$.  If $m\geq m_0$ is any positive
integer such that $m\Delta$ is integral, then
$$
(X,\Delta'=\frac {m-m_0}m\Delta+\frac 1 m A)
$$
is a simple normal crossings pair and it is kawamata log terminal.  Further
$$
K_X+\Delta'\sim _{\mathbb Q}K_X+\Delta+H/m,
$$
and $\Delta'$ is big over $T$.  Thus (2) follows from (1).

Note that 
$$
\vol(X_t,K_{X_t}+\Delta_t)=\lim_{\epsilon\to 0}\vol(X_t,K_{X_t}+(1-\epsilon)\Delta_t).
$$
As $(X_t,(1-\epsilon)\Delta_t)$ is kawamata log terminal, (3) follows from (1) and (2).  
\end{proof}

\section{DCC for the volume of bounded pairs}
\label{s_bound-volume}

We prove \eqref{t_dcc} in this section.  We first deal with the case that $T$ is a closed
point.

\begin{proposition}\label{p_limit} Fix a set $I\subset [0,1]$ which satisfies the 
DCC and a simple normal crossings pair $(Z,B)$, where $Z$ is projective of dimension $n$.
Let $\mathfrak{D}$ be the set of simple normal crossings pairs $(X,\Delta)$, where $X$ is
projective, the coefficients of $\Delta$ belong to $I$ and there is a birational morphism
$f\colon\map X.Z.$ with $\Phi=f_*\Delta\leq B$.

Then the set 
$$
\{\, \vol(X,K_X+\Delta) \,|\, (X,\Delta)\in \mathfrak{D}\,\},
$$
also satisfies the DCC.  
\end{proposition}

\begin{definition}\label{d_associated-b-divisor} Let $(X,\Delta)$ be a log pair.  
If $\pi\colon\map Y.X.$ is a birational morphism, then we may write
$$
K_Y+\Gamma=\pi^*(K_X+\Delta)+E,
$$
where $\Gamma\geq 0$ and $E\geq 0$ have no common components, $\pi_*\Gamma=\Delta$ and
$\pi_*E=0$.    

Define a \textbf{b}-divisor $\mathbf{L}_{\Delta}$ by setting
$\mathbf{L}_{\Delta,Y}=\Gamma$.
\end{definition}

\begin{lemma}\label{l_simple} Let $(X,\Delta)$ be a simple normal crossings pair, where $X$ is a
projective variety.
\begin{enumerate} 
\item If $\map Y.X.$ is a birational morphism such that $(Y,\Theta=\mathbf{L}_{\Delta,Y})$
has simple normal crossings, and $\Gamma-\Theta\geq 0$ is exceptional, then
$$
\vol(X,K_X+\Delta)=\vol(Y,K_Y+\Gamma).
$$
\item If $f\colon\map X.Z.$ is a birational morphism such that
$(Z,\Phi=\mathbf{L}_{\Delta,Z})$ has simple normal crossings and $\Theta=\Delta\wedge
\mathbf{L}_{\Phi,X}$, then
$$
\vol(X,K_X+\Delta)=\vol(X,K_X+\Theta).
$$
\end{enumerate} 
\end{lemma}
\begin{proof} (1) is clear, as 
$$
H^0(X,\ring X.(m(K_X+\Delta)))\simeq H^0(Y,\ring Y.(m(K_Y+\Gamma))),
$$
for all $m$.  

For (2), suppose $m$ is a sufficiently divisible positive integer.  We have
\begin{align*} 
H^0(X,\ring X.(m(K_X+\Delta)))&\subset H^0(Z,\ring Z.(m(K_Z+\Phi)))\\
                              &= H^0(X, \ring X.(m(K_X+\mathbf{L}_{\Phi,X}))),
\end{align*} 
and so
\[
H^0(X, \ring X.(m(K_X+\Delta)))=H^0(X,\ring X.(m(K_X+\Theta))).\qedhere
\]
\end{proof}

\begin{lemma}\label{d_toric} Let $(Z,\Phi)$ be a simple normal crossings pair 
which is log canonical.

If $\nu$ is a valuation such that $\mathbf{L}_{\Phi}(\nu)>0$, then the centre of $\nu$ is
a stratum $W$ of $(Z,\Phi)$ and there is a birational morphism $\map Y=Y_{\nu}.Z.$ such
that $\rho(Y/Z)=1$, $Y$ is $\mathbb{Q}$-factorial and the centre of $\nu$ is a divisor on
$Y$; $Y_{\nu}$ is unique with these properties.  
\end{lemma}
\begin{proof} This is a consequence of the existence of log terminal models, which is
proved in \cite{BCHM06}, and uniqueness of log canonical models.
\end{proof}

\begin{definition}\label{d_one-b-divisor} Let $(X,\Delta)$ be a log pair.  
Define a \textbf{b}-divisor $\mathbf{M}_{\Delta}$ by assigning to any valuation $\nu$, 
$$
\mathbf{M}_{\Delta}(\nu)=\begin{cases} \mult_B(\Delta) & \text{if the centre of $\nu$ is a divisor $B$ on $X$,} \\
                                       1 & \text{otherwise.}
\end{cases}
$$
\end{definition}

\begin{definition}\label{d_reduction} Let $\mathbf{B}$ be a $\mathbf{b}$-divisor whose 
coefficients belong to $[0,1]$ and let $(Z,\Phi=\mathbf{B}_Z)$ be a model with simple
normal crossings.  Let $\map Z'.Z.$ be a log resolution, and let $\Sigma$ be a set of
valuations $\sigma$ whose centres are exceptional divisors for $\map Z'.Z.$, such that
$\mathbf{L}_{\Phi}(\sigma)>0$.

For every valuation $\sigma\in \Sigma$, let $\Gamma_{\sigma}=(\mathbf{L}_{\Phi}\wedge
\mathbf{B})_{Y_{\sigma}}$, where $\map Y_{\sigma}.Z.$ is defined in \eqref{d_toric}.  Let
$$
\Theta=\bigwedge_{\sigma\in \Sigma} \mathbf{L}_{\Gamma_{\sigma},Z'},
$$
the minimum of the divisors $\mathbf{L}_{\Gamma_{\sigma},Z'}$.  

The \textbf{cut} of $(Z,\mathbf{B})$, associated to $\map Z'.Z.$ and $\Sigma$, is the
pair $(Z',\mathbf{B}')$, where
$$
\mathbf{B}'=\mathbf{B}\wedge \mathbf{M}_{\Theta},
$$ 
so that the trace of $\mathbf{B}'$ on $Z'$ is $\Theta\wedge \mathbf{B}_{Z'}$ and otherwise
$\mathbf{B}'$ is the same $\mathbf{b}$-divisor as $\mathbf{B}$.

We say that the pair $(Z',\mathbf{B}')$ is a \textbf{reduction} of the pair
$(Z,\mathbf{B})$, if they are connected by a sequence of cuts, that is, there are pairs,
$(Z_i,\mathbf{B}_i)$, $0\leq i\leq k$, starting at $(Z_0,\mathbf{B}_0)=(Z,\mathbf{B})$ and
ending at $(Z_k,\mathbf{B}_k)=(Z',\mathbf{B}')$, such that $(Z_{i+1},\mathbf{B}_{i+1})$ is
a cut of $(Z_i,\mathbf{B}_i)$, for each $0\leq i<k$.
\end{definition}

\begin{lemma}\label{l_reduction} Let $\mathbf{B}$ be a $\mathbf{b}$-divisor
whose coefficients belong to a set $I\subset [0,1]$ which satisfies the DCC, and let
$(Z,\Phi=\mathbf{B}_Z)$ be a model with simple normal crossings.

Then there is a reduction $(Z',\mathbf{B}')$ of $(Z,\mathbf{B})$ such that
$$
\mathbf{L}_{\Phi'}\leq \mathbf{B}',
$$ 
where $\Phi'=\mathbf{B}'_{Z'}$.  
\end{lemma}
\begin{proof} If $W$ is a stratum of $(Z,\Phi)$, then define the \textbf{weight} $w$ of
$W$ as follows:

\noindent If there is a valuation $\nu$, with centre $W$, such that
$\mathbf{B}(\nu)<\mathbf{L}_{\Phi}(\nu)$, then let $w$ be the number of components of
$\Phi$ with coefficient $1$ which contain $W$.  Otherwise, if there is no such $\nu$, then
let $w=-1$.  

Define the \textbf{weight} of $(Z,\mathbf{B})$ to be the maximum weight of the strata of
$(Z,\Phi)$.

Suppose the weight of $(Z,\mathbf{B})$ is $-1$.  Then $\mathbf{L}_{\Phi}(\nu)\leq
\mathbf{B}(\nu)$ for any valuation $\nu$ whose centre is a stratum.  If $\rho$ is a
valuation, whose centre is not a stratum, we have $0=\mathbf{L}_{\Phi}(\rho)\le
\mathbf{B}(\rho)$ (cf. \cite[2.31]{KM98}).  In this case we just take $Z'=Z$.

From now on we suppose that the weight $w\geq 0$.  Suppose that $(Z',\mathbf{B}')$ is a
cut of $(Z,\mathbf{B})$.  Then $\mathbf{B}'$ and $\mathbf{B}$ have the same coefficients,
except for finitely many valuations.  In particular, the coefficients of $\mathbf{B}'$
belong to a set $I'\supset I$ which still satisfies the DCC.  It suffices therefore to
prove that we can find a cut $(Z',\mathbf{B}')$ of $(Z,\mathbf{B})$ with smaller weight.

Now if $(Z',\mathbf{B}')$ is a cut of $(Z,\mathbf{B})$, then $\mathbf{B}'_{Z'}\leq
\mathbf{L}_{\Phi,Z'}$.  On the other hand, if $\nu$ is any valuation whose centre is not a
divisor on $Z'$, then $\mathbf{B}(\nu)=\mathbf{B}'(\nu)$.  It follows that the weight of
$(Z',\mathbf{B}')$ is at most the weight of $(Z,\mathbf{B})$.  Therefore, as $(Z,\Phi)$
has only finitely many strata, we may construct $(Z',\mathbf{B}')$ \'etale locally about
every stratum.  Thus, we may assume that $Z=\mathbb{C}^n$ and that $\Phi$ is supported on
the coordinate hyperplanes.  

We will use the language of toric geometry, cf. \cite{Fulton93}.  $\mathbb{C}^n$ is the
toric variety associated to the cone spanned by the standard basis vectors $\llist e.n.$
in $\mathbb{R}^n$.  If $\nu$ is any valuation such that $\mathbf{L}_{\Phi}(\nu)>0$, then
$\nu$ is toric and we will identify $\nu$ with an element $(\llist v.n.)$ of
$\mathbb{N}^n$.  Order the components of $\Phi$ so that the last $w$ components have
coefficient one and let $0\leq \llist c.s.<1$ be the initial coefficients, so that
$n=s+w$.  With this ordering, we have
$$
\mathbf{L}_{\Phi}(\nu)=1-\sum_i v_i(1-c_i).
$$
(Indeed both sides of this equation are affine linear in $\llist v.n.$ and $\llist c.s.$
and it is easy to check we have equality when either $\nu$ is the zero vector or when
$\nu=e_i$, $1\leq i\leq n$.)  Consider the finite set
$$
\mathfrak{F}=\{\, (\llist v.s.)\in \mathbb{N}^s \,|\, \sum_i v_i(1-c_i)<1 \,\}.
$$
Given a valuation $\nu=(\llist v.n.)$, note that $\mathbf{L}_{\Phi}(\nu)>0$ if and only if
$(\llist v.s.)\in\mathfrak{F}$.

As $I$ satisfies the DCC, for every $f=(\llist f.s.)\in \mathfrak{F}$, we may pick a
valuation $\sigma=(\llist f.s.,v_{s+1}, v_{s+2}, \dots, v_n)$, such that
$$
\mathbf{B}(\sigma)=\inf \{\, \mathbf{B}(\nu) \,|\, \nu=(\llist f.s.,u_{s+1}, u_{s+2},\dots, u_n) \,\}.
$$
Let $\Sigma$ be a set of choices of such valuations $\sigma$, so that $\Sigma$ and
$\mathfrak{F}$ have the same cardinality.  Let $\map Z'.Z.$ be any log resolution of
$(Z,\Phi)$ such that the centre of every element of $\Sigma$ is a divisor on $Z'$.  We may
assume that the induced birational map $\map Z'.Y_{\sigma}.$ is a morphism, for every
$\sigma\in \Sigma$.  Let $(Z',\mathbf{B}')$ be the cut of $(Z,\mathbf{B})$ associated
to $Z'\to Z$ and $\Sigma$.

There are two cases.  If $w=0$, then $\Sigma=\mathfrak{F}$ is the set of all valuations of
coefficient less than one.  It follows that if $\nu$ is any valuation whose centre on $Z'$
is not a divisor, then $\mathbf{L}_{\Phi'}(\nu)=0$, so that the weight of
$(Z',\mathbf{B}')$ is $-1$, which is less than the weight of $(Z,\mathbf{B})$.

Otherwise we may assume that $w\geq 1$.  Suppose that $\nu$ is a valuation whose centre is
not a divisor on $Z'$ such that $\mathbf{B}'(\nu)<1$ and $\mathbf{L}_{\Phi'}(\nu)>0$.
Then $\mathbf{B}(\nu)=\mathbf{B}'(\nu)$ and $\mathbf{L}_{\Phi}(\nu)>0$ and so $\nu=(\llist
v.n.)$ is toric and $(\llist v.s.)\in \mathfrak{F}$.  By construction, there is an element
$\sigma$ of $\Sigma$ with the same first $s$ coordinates as $\nu$ such that
$\mathbf{B}(\sigma)\leq \mathbf{B}(\nu)<1$.  The cone spanned by the standard basis
vectors $\llist e.n.$ is divided into $m\leq n$ subcones by $\sigma$ (these are the
maximal cones of $Y_{\sigma}$), where $m$ is the number of non-zero entries of $\sigma$,
and so $\nu$ is a non-negative linear combination of $\sigma$ and $n-1$ vectors taken from
$\llist e.n.$.  It follows that
\[
\label{e_simple}\nu=\sum_{j\neq l} \lambda_je_j+\lambda \sigma,\tag{$\sharp$}
\]
for some index $1\leq l\leq n$ and non-negative real numbers $\llist \lambda.n.$ and
$\lambda$, where the $l$th entry of $\sigma$ is non-zero.  

If the centre of $\nu$ on $Z'$ is contained in $w$ components of $\Phi'$ of coefficient
one, then the centre of $\nu$ on $Y_{\sigma}$ is also contained in $w$ components of
coefficient one of $\Gamma_{\sigma}=(\mathbf{L}_{\Phi}\wedge \mathbf{B})_{Y_{\sigma}}$.
By assumption, the exceptional divisor of $\map Y_{\sigma}.Y.$ has coefficient strictly
less than one, and so the centre of $\nu$ on $Y_{\sigma}$ must be contained in the strict
transform of the last $w$ coordinate hyperplanes.  But then, by standard toric geometry,
$l\leq s$.  Hence the $l$th entry of $\nu$ is non-zero.  Comparing the coefficient of
$e_l$ in \eqref{e_simple}, we must have $\lambda=1$, so that $\nu\geq \sigma$.  In this
case,
\begin{align*} 
\mathbf L _{\Phi'}(\nu) 
&\leq\mathbf{L}_{\Gamma_{\sigma}}(\nu)    && \text{by definition of $\mathbf{B}'$,}\\
&\leq\mathbf{L}_{\Gamma_{\sigma}}(\sigma) &&\text{as $\nu\geq\sigma$,}\\
&\leq \mathbf{B}(\sigma)                &&\text{since $\Gamma_{\sigma}\leq\mathbf{B}_{Y_{\sigma}}$,}\\
&\leq\mathbf{B}(\nu)                    &&\text{by our choice of $\sigma$,}\\
&=\mathbf{B'}(\nu)                      &&\text{by definition of $\mathbf{B'}$.}
\end{align*} 
It follows that the weight of $(Z',\mathbf{B}')$ is indeed smaller than the weight of
$(Z,\mathbf{B})$ and this completes the induction and the proof.
\end{proof}

\begin{proof}[Proof of \eqref{p_limit}] Suppose we have a sequence of log pairs
$(X_i,\Delta_i)\in \mathfrak{D}$, such that $v_i\geq v_{i+1}$, where
$v_i:=\vol(X_i,K_{X_i}+\Delta_i)$.  We will show that the sequence $\ilist v.$ is
eventually constant; to this end we are free to pass to a subsequence.  Replacing $I$ by
$\bar{I}\cup \{1\}$, we may assume that $I$ is closed and $1\in I$.

By assumption there are projective birational morphisms $f_i\colon\map X_i.Z.$ such that
$\Phi_i=f_{i*}\Delta_i\leq B$.  Note that if $\nu$ is a valuation such that
$\mathbf{M}_{\Delta_i}(\nu)\notin\{0,1\}$, then the centre of $\nu$ is a component of
$\Delta_i$.  On the other hand, if $\mathbf{M}_{\Delta_i}(\nu)=1$ and
$\mathbf{M}_{\Delta_j}(\nu)=0$ and the centre of $\nu$ is not a component of $\Delta_i$,
then the centre of $\nu$ is not a divisor on $X_i$ and it is a divisor on $X_j$.
Therefore there are only countably many valuations $\nu$ such that
$\mathbf{M}_{\Delta_i}(\nu)\neq \mathbf{M}_{\Delta_j}(\nu)$ for some $i$ and $j$.
Therefore, as $I$ satisfies the DCC, by a standard diagonalisation argument, after passing
to a subsequence, we may assume that $\mathbf{M}_{\Delta_i}(\nu)$ is eventually a
non-decreasing sequence, for all valuations $\nu$.  In particular, we may define a
\textbf{b}-divisor $\mathbf{B}$ by putting
$$
\mathbf{B}(\nu)=\lim_{i\to\infty}\mathbf{M}_{\Delta_i}(\nu).
$$
Note that the coefficients of $\mathbf{B}$ belong to $I$.  Let $\Phi=\mathbf{B}_Z$. 

Suppose that $(Z',\mathbf{B}')$ is a cut of $(Z,\mathbf{B})$ associated to a birational
morphism $\map Z'.Z.$ and a set of valuations $\Sigma$.  Let $f'_i\colon\rmap X_i.Z'.$ be
the induced birational map.  Note that, if $\map X'_i.X_i.$ is a birational morphism and
$(X'_i,\Delta'_i=\mathbf{M}_{\Delta_i,X'_i})$ has simple normal crossings, then
$\vol(X'_i,K_{X'_i}+\Delta'_i)=v_i$ and the coefficients of $\Delta'_i$ belong to $I$, so
that $(X'_i,\Delta'_i)\in \mathfrak{D}$.  Therefore, we are free to replace
$(X_i,\Delta_i)$ by $(X'_i,\Delta'_i)$.  In particular, we may assume that $f'_i$ is a
birational morphism.

Given $\sigma\in \Sigma$, let $\Gamma_{i,\sigma}=(\mathbf{L}_{\Phi_i}\wedge
\mathbf{B})_{Y_{\sigma}}$, where $\map Y_{\sigma}.Z.$ is defined in \eqref{d_toric}.
Suppose we define a sequence of divisors
$$
\Theta_i=\bigwedge_{\sigma\in \Sigma} \mathbf{L}_{\Gamma_{i,\sigma},Z'},
$$
as in \eqref{d_reduction}.  Suppose that $B$ is a prime divisor on $Z'$ which is exceptional over $Z$.  Then the coefficient of $B$ in $\Theta_i$ is the minimum
of finitely many affine linear functions of the coefficients of $\Delta_i$.  It follows
that the coefficients of $\ilist \Theta.$ belong to a set $I'\supset I$ which satisfies
the DCC.  Finally, let
$$
\Delta'_i=\Delta_i\wedge \mathbf{M}_{\Theta_i,X_i},
$$
so that we only change the coefficients of divisors which are exceptional for $Z'\to Z$.
In particular, the coefficients of $\Delta_i'$ belong to $I'$.  On the other hand,
$$
\Delta_i\wedge \mathbf{L}_{\Theta_i,X_i} \leq \Delta'_i=\Delta_i\wedge \mathbf{M}_{\Theta_i,X_i}\leq \Delta_i,
$$
so that 
$$
v_i=\vol(X_i,K_{X_i}+\Delta'_i), 
$$
by (2) of \eqref{l_simple}.  Finally, note that $\mathbf{M}_{\Delta'_i}(\rho)$ is
eventually a non-decreasing sequence for any valuation $\rho$.  In particular, we may
define a \textbf{b}-divisor $\mathbf{B}'$ by putting
$$
\mathbf{B}'(\rho)=\lim_{i\to\infty}\mathbf{M}_{\Delta'_i}(\rho),
$$
as before. 

Hence, \eqref{l_reduction} implies that we may find a reduction $(Z',\mathbf{B}')$, of
$(Z,\mathbf{B})$ and pairs $(X'_i,\Delta'_i)$, whose coefficients belong to a set $I'$
which satisfies the DCC, such that $v_i=\vol(X'_i,K_{X'_i}+\Delta'_i)$, there is a
birational morphism $\map X'_i.Z'.$, and moreover $\mathbf{L}_{\Phi'}\leq \mathbf{B}'$.
Replacing $(X_i,\Delta_i)$ by $(X'_i,\Delta'_i)$, $I$ by $I'$, and $Z$ by $Z'$, we may
therefore assume that $\mathbf{L}_{\Phi}\leq \mathbf{B}$.

Note that
\begin{align*} 
v_i&=\vol(X_i,K_{X_i}+\Delta_i) \\
   &\leq \vol(Z,K_Z+\Phi_i) \\
   &\leq \lim \vol(Z,K_Z+\Phi_i) \\
   &=\vol(Z,K_Z+\Phi),
\end{align*} 
as $\lim \Phi_i=\Phi$.  

On the other hand, if we fix $\epsilon>0$, then $(Z,(1-\epsilon)\Phi)$ is kawamata log
terminal.  In particular, we may pick a birational morphism $f\colon\map Y.Z.$ such that
$(Y,\Psi=\mathbf{L}_{(1-\epsilon)\Phi,Y})$ is terminal.  If we let
$\Theta=\mathbf{L}_{\Phi,Y}$ and $\Gamma=\mathbf{B}_Y$, then
$$
\Psi\leq (1-\eta)\Theta\leq \Theta\leq \Gamma
$$
for some $\eta>0$.  As $\Gamma$ is the limit of $\Gamma_i=\mathbf{M}_{\Delta_i,Y}$, it
follows that we may find $i$ such that $\Psi\leq \Gamma_i$.  As $(Y,\Psi)$ is terminal, we
have $\Psi_i=\mathbf{L}_{\Psi,X_i}\leq \Delta_i$.  But then
\begin{align*} 
\vol(Z,K_Z+(1-\epsilon)\Phi)&=\vol(Y,K_Y+\Psi)\\
                            &\leq \vol(X_i,K_{X_i}+\Psi_i) \\
                            &\leq \vol(X_i,K_{X_i}+\Delta_i)=v_i.
\end{align*} 
Taking the limit as $\epsilon$ goes to zero, we get 
$$
\vol(Z,K_Z+\Phi)\leq v_i\leq \vol(Z,K_Z+\Phi),
$$
so that $v_i=\vol(Z,K_Z+\Phi)$ is constant.  
\end{proof}

\begin{proof}[Proof of \eqref{t_dcc}] We may assume that $1\in I$.  By assumption there is
a log pair $(Z,B)$ and a projective morphism $\map Z.T.$, where $T$ is of finite type,
such that if $(X,\Delta)\in \mathfrak{D}$, then there is a closed point $t\in T$ and a
birational map $f\colon\rmap X.Z_t.$ such that the support of $B_t$ contains the support
of the strict transform of $\Delta_t$ and any $f^{-1}$-exceptional divisor.

Suppose that $p\colon\map Y.X.$ is a birational morphism.  Then the coefficients of
$\Gamma=\mathbf{M}_{\Delta,Y}$ belong to $I$ and
$$
\vol(X,K_X+\Delta)=\vol(Y,K_Y+\Gamma), 
$$
by (1) of \eqref{l_simple}.  Replacing $(X,\Delta)$ by $(Y,\Gamma)$, we may assume that
$f$ is a morphism and we are free to replace $Z$ and $B$ by higher models.

We may assume that $T$ is reduced.  Blowing up and decomposing $T$ into a finite union of
locally closed subsets, we may assume that $(Z,B)$ has simple normal crossings; passing to
an open subset of $T$, we may assume that the fibres of $\map Z.T.$ are log pairs, so that
$(Z,B)$ has simple normal crossings over $T$; passing to a finite cover of $T$, we may
assume that every stratum of $(Z,B)$ has irreducible fibres over $T$; decomposing $T$ into
a finite union of locally closed subsets, we may assume that $T$ is smooth; finally
passing to a connected component of $T$, we may assume that $T$ is integral.

Let $Z_0$ and $B_0$ be the fibres over a fixed closed point $0\in T$.  Let
$\mathfrak{D}_0\subset \mathfrak{D}$ be the set of simple normal crossings pairs
$(Y,\Gamma)$, where the coefficients of $\Gamma$ belong to $I$, $Y$ is a projective
variety of dimension $n$, and there is a birational morphism $g\colon\map Y.Z_0.$ with
$g_*\Gamma\leq B_0$.

Pick $(X,\Delta)\in \mathfrak{D}$.  Let $\Phi=f_*\Delta$.  Let $\Sigma$ be the set of all
valuations $\nu$ whose centre on $X$ is a divisor which is exceptional over $Z_t$ such that
$\mathbf{L}_{\Phi}(\nu)>0$.  We may find a birational morphism $f'\colon\map X'.Z_t.$,
such that the centre of every element of $\Sigma$ is a divisor on $X'$, whilst $f'$ only
blows up strata of $(Z_t,\Phi)$.  Suppose that $g\colon\map W.X.$ is a log resolution
which resolves the indeterminacy locus of the induced birational map $\rmap X.X'.$.  If
we set $\Delta'=\mathbf{M}_{\Delta,X'}$, then the coefficients of $\Delta'$ belong to $I$
and
$$
\vol(X,K_X+\Delta)=\vol(W,K_W+\mathbf{M}_{\Delta,W})\leq \vol(X',K_{X'}+\Delta'),
$$
by (1) of \eqref{l_simple}.  If $\nu$ is any valuation whose centre is an exceptional divisor for
$\map W.X'.$ but not for $\map W.X.$, then the centre of $\nu$ is an exceptional divisor for $\map
X.Z_t.$ and so $\mathbf{L}_{\Phi}(\nu)=0$, by choice of $f'$.
It follows that
$$
\mathbf{M}_{\Delta,W}\geq \mathbf{M}_{\Delta',W}\wedge \mathbf{L}_{\Phi,W},
$$
and so \eqref{l_simple} implies that
$$
\vol(W,K_W+\mathbf{M}_{\Delta,W})\geq \vol(W,K_W+\mathbf{M}_{\Delta',W}\wedge \mathbf{L}_{\Phi,W})=\vol(X',K_{X'}+\Delta').
$$
and hence the inequalities above are equalities. In particular
$$
\vol(X,K_X+\Delta)=\vol(X',K_{X'}+\Delta').
$$
Replacing $(X,\Delta)$ by $(X',\Delta')$, we may assume that $f$ only blow ups strata of
$\Phi$.

As $(Z,B)$ has simple normal crossings over $T$ and the strata of $(Z,B)$ have irreducible
fibres, we may find a sequence of blow ups $g\colon\map Z'.Z.$ of strata of $B$, which
induces the sequence of blow ups determined by $f$, so that $X=Z'_t$.  There is a unique
divisor $\Psi$ supported on the strict transform of $B$ and the exceptional locus of $g$,
such that $\Delta=\Psi_t$.  If $Y=Z'_0$ is the fibre over $0$ of $\map Z'.T.$ and $\Gamma$
is the restriction of $\Psi$ to $Y$, then $(Y,\Gamma)\in \mathfrak{D}_0$.  \eqref{t_dilp}
implies that $\vol(Y,K_Y+\Gamma)=\vol(X,K_X+\Delta)$.

It follows that 
$$
\{\, \vol(X,K_X+\Delta) \,|\, (X,\Delta)\in \mathfrak{D} \,\}=\{\, \vol(X,K_X+\Delta) \,|\, (X,\Delta)\in \mathfrak{D}_0 \,\}.
$$
Now apply \eqref{p_limit}.  \end{proof}

\section{Birational geometry of global quotients}
\label{s_quotient}

\begin{theorem}[Tsuji]\label{t_tsuji} Assume \eqref{t_volume}$_{n-1}$.  

Then there is a constant $C=C(n)>2$ such that if $(X,\Delta)$ is a global quotient, where
$X$ is projective of dimension $n$, and $K_X+\Delta$ is big, then $\phi_{m(K_X+\Delta)}$
is birational for every integer $m\geq C+1$ such that
$$
\vol(X,(m-1)(K_X+\Delta))> (Cn)^n.
$$
\end{theorem}
\begin{proof} First note that \eqref{l_round} implies that 
$$
K_X+\rfup (m-1)(K_X+\Delta).=\rfdown m(K_X+\Delta)..
$$
As we are assuming \eqref{t_volume}$_{n-1}$ there is a constant $\epsilon>0$ such that if
$(U,\Theta)$ is a global quotient, where $K_U+\Theta$ is big and $U$ is projective of
dimension $k$ at most $n-1$, then $\vol(U,K_U+\Theta)>\epsilon^k$.  Let
$$
C=2(1+\gamma)^{n-1}\qquad \text{where} \qquad \gamma=\frac {4n}{\epsilon}.
$$

By assumption there is a smooth projective variety $Y$ of dimension $n$ and a finite group
$G\subset \Aut(Y)$ such that $X=Y/G$ and if $\pi\colon\map Y.X.$ is the quotient morphism,
then $K_Y=\pi^*(K_X+\Delta)$.  As $K_X+\Delta$ is big, $Y$ is of general type.  Replacing
$(X,\Delta)$ and $Y$ by their log canonical models, which exist by \cite{BCHM06}, we lose
the fact that $X$ and $Y$ are smooth, gain the fact that $K_X+\Delta$ and $K_Y$ are ample,
and retain the condition that $K_X+\Delta$ is kawamata log terminal and $K_Y$ is
canonical.

We check the hypotheses of \eqref{t_recursive}, applied to the ample divisor $K_X+\Delta$
and the constants $\epsilon/2$ and $\gamma_0=\frac{m-1}C\geq 1$.  Clearly
$$
\vol(X,\gamma_0(K_X+\Delta))>n^n.
$$
Suppose that $V$ is a minimal non kawamata log terminal centre of a log pair
$(X,\Delta+\Delta_0)$, which is log canonical at the generic point of $V$.  Further
suppose that $V$ passes through a very general point of $X$, and $0\leq \Delta_0
\sim_{\mathbb{Q}}\lambda (K_X+\Delta)$, for some rational number $\lambda\geq 1$.

If $\Gamma_0=\pi^*\Delta_0$, then every irreducible component of $\pi^{-1}(V)$ is a non
kawamata log terminal centre of $(Y,\Gamma_0)$.  Let $V'$ be the normalisation of
$\pi^{-1}(V)$.  As $H=K_Y+\Gamma_0$ is ample, Kawamata's subadjunction formula implies
that for every $\eta>0$, there is a divisor $\Phi\geq 0$ on $V'$ such that
$$
(K_Y+\Gamma_0+\eta H)|_{V'}=K_{V'}+\Phi.
$$
Let $\map W.V'.$ be a $G$-equivariant resolution.  As $V$ passes through a very general
point of $X$, $W$ is a union of irreducible varieties of general type.  If $U=W/G$ is the
quotient, then $U$ is irreducible and we may find a $\mathbb{Q}$-divisor $\Theta$ such
that $K_W=\psi^*(K_U+\Theta)$, where $\psi\colon\map W.U.$ is the quotient map.

As $(U,\Theta)$ is a global quotient, $\vol(U,K_U+\Theta)>\epsilon^k$, where $k$ is the
dimension of $V$.  Therefore
\begin{align*} 
|G| \vol(V,(K_X+\Delta+\Delta_0)|_V) &=\vol(V',(K_Y+\Gamma_0)|_{V'})\\ 
                                 &\geq \vol(V',K_{V'})\\ 
                                 &\geq \vol(W,K_W)\\ 
                                 &=|G|\vol(U,K_U+\Theta)\\ 
                                 &\geq |G|\epsilon^k.
\end{align*} 
Thus 
$$
\vol(V,(1+\lambda)(K_X+\Delta)|_V)>\epsilon^k,
$$
and so
$$
\qquad \vol(V,\lambda(K_X+\Delta)|_V)>\left (\frac{\epsilon}{2}\right )^k.  
$$
\eqref{t_recursive} implies that $(m-1)(K_X+\Delta)$ is potentially birational.  (1) of
\eqref{l_potential} implies that $\phi_{K_X+\rfup (m-1)(K_X+\Delta).}$ is
birational.  \end{proof}

\section{Proof of \eqref{t_volume}  and \eqref{t_boundauto} }
\label{s_proofs}

\begin{proof}[Proof of \eqref{t_volume}] By induction on $n$.  Assume
\eqref{t_volume}$_{n-1}$.  By \eqref{t_tsuji} there is a constant $C=C(n)>2$ depending
only on the dimension $n$ such that if $(X,\Delta)$ is a global quotient, where $X$ is
projective of dimension $n$ and $K_X+\Delta$ is big, then $\phi_{m(K_X+\Delta)}$ is
birational, for any $m\geq C+1$ such that
$$
\vol(X,(m-1)(K_X+\Delta))> (Cn)^n.  
$$
Note that the right hand side does not depend on $m$.   

Fix a constant $V>n^n$ and let
$$
\mathfrak{D}_V=\{\, (X,\Delta)\in \mathfrak{D} \,|\, 0<\vol(X,K_X+\Delta)\leq V \,\}.
$$
Note that if $k$ is a positive integer such that $\vol(X,k(K_X+\Delta))\leq C^nV$, then
$\vol(X,(k+1)(K_X+\Delta))\leq 2^nC^nV$.  It follows that there is a positive integer
$m\geq C+1$ such that if $(X,\Delta)\in \mathfrak{D}_V$, then
$$
(Cn)^n < \vol(X,(m-1)(K_X+\Delta))\leq 2^nC^nV, 
$$
so that $\phi_{m(K_X+\Delta)}$ is birational.  (2) of \eqref{l_potential} implies that
$\phi_{K_X+(2n+1)m(K_X+\Delta)}$ is birational.  But then \eqref{t_boundingimage} implies
that $\mathfrak{D}_V$ is log birationally bounded, and so \eqref{t_dcc} implies that the
set
$$
\{\, \vol(X,K_X+\Delta) \,|\, (X,\Delta)\in \mathfrak{D}_V \,\},
$$
satisfies the DCC, which implies that (1) and (2) of \eqref{t_volume} hold 
in dimension $n$.  

In particular there is a constant $\delta>0$ such that if $(X,\Delta)\in \mathfrak{D}$,
and $K_X+\Delta$ is big, then $\vol(X,K_X+\Delta)\geq \delta$.  It follows that
$\phi_{M(K_X+\Delta)}$ is birational, for any
$$
M> \frac {Cn}{\delta}+1,
$$
and this completes the induction and the proof.  \end{proof}

\begin{proof}[Proof of \eqref{t_boundauto}] By (2) of \eqref{t_volume} there is a constant
$\delta>0$ such that if $(X,\Delta)$ is a global quotient, where $X$ is projective of
dimension $n$ and $K_X+\Delta$ is big, then $\vol(X,K_X+\Delta)\geq \delta$.  Let $c=\frac
1{\delta}$.

Let $Y$ be a projective variety of dimension $n$ of general type.  By \cite{BCHM06}, there
is a log canonical model $\rmap Y.Y'.$.  If $G$ is the birational automorphism group of
$Y$, then $G$ is the automorphism group of $Y'$.  Replacing $Y$ by a $G$-equivariant
resolution of $Y'$, we may assume that $G$ is the automorphism group of $Y$.  Let
$\pi\colon\map Y.X=Y/G.$ be the quotient of $Y$.  Then there is a divisor $\Delta$ on $X$
such that $K_Y=\pi^*(K_X+\Delta)$.  By definition, $(X,\Delta)$ is a global quotient, $X$
is projective and $K_X+\Delta$ is big.  It follows that $\vol(X,K_X+\Delta)\geq \delta$.  
As 
$$
\vol(Y,K_Y)=|G|\vol(X,K_X+\Delta),
$$
it follows that 
\[
|G|\leq c\cdot \vol(Y,K_Y).\qedhere
\]
\end{proof}

\bibliographystyle{hamsplain}
\bibliography{/home/mckernan/Jewel/Tex/math}

\providecommand{\bysame}{\leavevmode\hbox to3em{\hrulefill}\thinspace}
\begin{thebibliography}{10}

\bibitem{Alexeev94}
V.~Alexeev, \emph{Boundedness and {$K^2$} for log surfaces}, International J.
  Math. \textbf{5} (1994), 779--810.

\bibitem{AM04}
V.~Alexeev and S.~Mori, \emph{Bounding singular surfaces of general type},
  Algebra, arithmetic and geometry with applications ({W}est {L}afayette, {IN},
  2000), Springer, Berlin, 2004, pp.~143--174.

\bibitem{Andreotti50}
A.~Andreotti, \emph{Sopra le superficie algebriche che posseggono
  trasformazioni birazionali in s\`e}, Univ. Roma Ist. Naz. Alta Mat. Rend.
  Mat. e Appl. (5) \textbf{9} (1950), 255--279.

\bibitem{AS95}
U.~Angehrn and Y.~Siu, \emph{Effective freeness and point separation for
  adjoint bundles}, Invent. Math. \textbf{122} (1995), no.~2, 291--308.

\bibitem{Ballico93}
E.~Ballico, \emph{On the automorphisms of surfaces of general type in positive
  characteristic}, Atti Accad. Naz. Lincei Cl. Sci. Fis. Mat. Natur. Rend.
  Lincei (9) Mat. Appl. \textbf{4} (1993), no.~2, 121--129.

\bibitem{BP10}
B.~Berndtsson and M.~P{\u{a}}un, \emph{Quantitative extensions of
  pluricanonical forms and closed positive currents}, Nagoya Math. J.
  \textbf{205} (2012), 25--65.

\bibitem{BCHM06}
C.~Birkar, P.~Cascini, C.~Hacon, and J.~M\textsuperscript{c}Kernan,
  \emph{Existence of minimal models for varieties of log general type}, J.
  Amer. Math. Soc. \textbf{23} (2010), no.~2, 405--468,
  \mbox{arXiv:math.AG/0610203}.

\bibitem{Cai00}
J.~Cai, \emph{Bounds of automorphisms of surfaces of general type in positive
  characteristic}, J. Pure Appl. Algebra \textbf{149} (2000), no.~3, 241--250.

\bibitem{CS95}
F.~Catanese and M.~Schneider, \emph{Polynomial bounds for abelian groups of
  automorphisms}, Compositio Math. \textbf{97} (1995), no.~1-2, 1--15, Special
  issue in honour of Frans Oort.

\bibitem{Corti91}
A.~Corti, \emph{Polynomial bounds for the number of automorphisms of a surface
  of general type}, Ann. Sci. \'Ecole Norm. Sup. (4) \textbf{24} (1991), no.~1,
  113--137.

\bibitem{FH11}
T.~de~Fernex and C.~Hacon, \emph{Deformations of canonical pairs and {F}ano
  varieties}, J. Reine Angew. Math. \textbf{651} (2011), 97--126.

\bibitem{Fulton93}
W.~Fulton, \emph{Introduction to toric varieties}, Princeton University Press,
  1993.

\bibitem{GLS94}
D.~Gorenstein, R.~Lyons, and R.~Solomon, \emph{The classification of the finite
  simple groups}, Mathematical Surveys and Monographs, vol.~40, American
  Mathematical Society, Providence, RI, 1994.

\bibitem{EGAIV}
A.~Grothendieck, \emph{{El\'em\'ents de G\'eom\'etrie Alg\'ebrique. {IV}.}},
  Publ. Math. IHES \textbf{28} (1966).

\bibitem{HM05b}
C.~Hacon and J.~M\textsuperscript{c}Kernan, \emph{Boundedness of pluricanonical
  maps of varieties of general type}, Invent. Math. \textbf{166} (2006), no.~1,
  1--25.

\bibitem{HS91}
A.~T. Huckleberry and M.~Sauer, \emph{On the order of the automorphism group of
  a surface of general type}, Math. Z. \textbf{205} (1990), no.~2, 321--329.

\bibitem{Kollar92b}
J.~Koll{\'a}r, \emph{Log surfaces of general type; some conjectures}, Contemp.
  Math. \textbf{162} (1994), 261--275.

\bibitem{Kollar95}
\bysame, \emph{Singularities of pairs}, Algebraic geometry---Santa Cruz 1995,
  Amer. Math. Soc., Providence, RI, 1997, pp.~221--287.

\bibitem{Kollar08}
\bysame, \emph{Is there a topological {B}ogomolov-{M}iyaoka-{Y}au inequality?},
  Pure Appl. Math. Q. \textbf{4} (2008), no.~2, part 1, 203--236.

\bibitem{KM98}
J.~Koll{\'a}r and S.~Mori, \emph{Birational geometry of algebraic varieties},
  Cambridge tracts in mathematics, vol. 134, Cambridge University Press, 1998.

\bibitem{Lazarsfeld04b}
R.~Lazarsfeld, \emph{Positivity in algebraic geometry. {II}}, Ergebnisse der
  Mathematik und ihrer Grenzgebiete. 3. Folge. A Series of Modern Surveys in
  Mathematics [Results in Mathematics and Related Areas. 3rd Series. A Series
  of Modern Surveys in Mathematics], vol.~49, Springer-Verlag, Berlin, 2004,
  Positivity for vector bundles, and multiplier ideals.

\bibitem{Matsumura63}
H.~Matsumura, \emph{On algebraic groups of birational transformations}, Atti
  Accad. Naz. Lincei Rend. Cl. Sci. Fis. Mat. Natur. (8) \textbf{34} (1963),
  151--155.

\bibitem{Nakayama04}
N.~Nakayama, \emph{Zariski-decomposition and abundance}, MSJ Memoirs, vol.~14,
  Mathematical Society of Japan, Tokyo, 2004.

\bibitem{Siu98}
Y-T. Siu, \emph{Invariance of plurigenera}, Invent. Math. \textbf{134} (1998),
  no.~3, 661--673.

\bibitem{Stichtenoth73}
H.~Stichtenoth, \emph{\"{U}ber die {A}utomorphismengruppe eines algebraischen
  {F}unktionenk\"orpers von {P}rimzahlcharakteristik. {I}. {E}ine
  {A}bsch\"atzung der {O}rdnung der {A}utomorphismengruppe}, Arch. Math.
  (Basel) \textbf{24} (1973), 527--544.

\bibitem{Szabo96}
E.~Szab{\'o}, \emph{Bounding automorphism groups}, Math. Ann. \textbf{304}
  (1996), no.~4, 801--811.

\bibitem{Takayama06}
S.~Takayama, \emph{Pluricanonical systems on algebraic varieties of general
  type}, Invent. Math. \textbf{165} (2006), no.~3, 551--587.

\bibitem{Tsuji00}
H.~Tsuji, \emph{{Bound of automorphisms of projective varieties of general
  type}}, \mbox{arXiv:math/0004138v1}.

\bibitem{Tsuji07}
\bysame, \emph{Pluricanonical systems of projective varieties of general type.
  {II}}, Osaka J. Math. \textbf{44} (2007), no.~3, 723--764.

\bibitem{Xiao94}
G.~Xiao, \emph{Bound of automorphisms of surfaces of general type. {I}}, Ann.
  of Math. (2) \textbf{139} (1994), no.~1, 51--77.

\bibitem{Xiao95}
\bysame, \emph{Bound of automorphisms of surfaces of general type. {II}}, J.
  Algebraic Geom. \textbf{4} (1995), no.~4, 701--793.

\bibitem{Xiao96}
\bysame, \emph{Linear bound for abelian automorphisms of varieties of general
  type}, J. Reine Angew. Math. \textbf{476} (1996), 201--207.

\bibitem{Zhang07}
De-Qi Zhang, \emph{Small bound for birational automorphism groups of algebraic
  varieties}, Math. Ann. \textbf{339} (2007), no.~4, 957--975, With an appendix
  by Yujiro Kawamata.

\end{thebibliography}


\end{document}